\documentclass[11pt,a4paper]{article}

\usepackage{amsmath, amssymb, amsthm} 
\usepackage{thmtools} 
\usepackage{xspace}
\usepackage{xcolor}
\usepackage{enumitem}
\usepackage{hyperref}
\usepackage{mathrsfs}
\usepackage{stackrel}
\usepackage{enumitem}
\usepackage{cleveref}
\usepackage{bbm}
\usepackage{varwidth}
\usepackage{tasks}

\renewcommand{\epsilon}{\varepsilon}
\newcommand{\bz}{\bold{z}}

\newcommand{\br}{\bold{r}}
\newcommand{\bx}{\bold{x}}
\newcommand{\by}{\bold{y}}
\newcommand{\bv}{\bold{v}}
\newcommand{\bu}{\bold{u}}
\newcommand{\supp}{\mathrm{Supp}\;}
\newcommand{\Diff}{\mathrm{Diff}}
\newcommand{\D}{\mathcal{D}}
\newcommand{\R}{\mathcal{R}}

\newcommand{\cS}{\mathcal{S}}
\newcommand{\DISC}{\mathrm{DISC}}
\newcommand{\la}{\langle}
\newcommand{\ra}{\rangle}

\numberwithin{equation}{section}
\crefformat{footnote}{#2\footnotemark[#1]#3}

\setlist[description]{font=\normalfont}
\allowdisplaybreaks
\newcommand{\sm}{\setminus}
\DeclareMathOperator{\Ex}{\mathbb{E}}
\renewcommand{\Pr}{\mathbb{P}}

\makeatletter
\renewcommand*{\eqref}[1]{%
  \hyperref[{#1}]{\textup{\tagform@{\ref*{#1}}}}%
}
\makeatother

\hypersetup{
    colorlinks,
    linkcolor={red!50!black},
    citecolor={blue!50!black},
    urlcolor={blue!80!black}
}

\def\QED{$\blacksquare$}
\def\inQED{$\square$}

\renewenvironment{proof}
{\vspace{1ex}\noindent{\bf Proof.}\hspace{0.5em}}{\hfill \QED \vspace{1ex}}

\newenvironment{proofof}[1]
{\vspace{1ex}\noindent{\bf Proof of #1.}\hspace{0.5em}}{\hfill \QED \vspace{1ex}}

\theoremstyle{plain}
\newtheorem{theorem}{Theorem}[section]
\newtheorem{lemma}[theorem]{Lemma}
\newtheorem{claim}[theorem]{Claim}
\newtheorem{proposition}[theorem]{Proposition}
\newtheorem{observation}[theorem]{Observation}

\newtheorem{remark}[theorem]{Remark}

\newtheorem{question}[theorem]{Question}

\makeatletter
\def\moverlay{\mathpalette\mov@rlay}
\def\mov@rlay#1#2{\leavevmode\vtop{%
   \baselineskip\z@skip \lineskiplimit-\maxdimen
   \ialign{\hfil$\m@th#1##$\hfil\cr#2\crcr}}}
\newcommand{\charfusion}[3][\mathord]{
    #1{\ifx#1\mathop\vphantom{#2}\fi
        \mathpalette\mov@rlay{#2\cr#3}
      }
    \ifx#1\mathop\expandafter\displaylimits\fi}
\makeatother

\let\theta=\vartheta
\let\rho=\varrho
\let\phi=\varphi

\title{{\bf Smoothed Analysis of the Koml\'os Conjecture: Rademacher Noise}}
\author{
	Elad Aigner-Horev \thanks{School of Computer Science, Ariel University, Ariel 40700, Israel. Email: {\tt horev@ariel.ac.il}.} 
	\and 
	Dan Hefetz \thanks{School of Computer Science, Ariel University, Ariel 40700, Israel. Email: {\tt danhe@ariel.ac.il}.}
	\and 
	Michael Trushkin \thanks{School of Computer Science, Ariel University, Ariel 40700, Israel. Email: {\tt michaelt@ariel.ac.il}.}
}

\begin{document}
\date{}
\maketitle

\begin{abstract}
The {\em discrepancy} of a matrix  $M \in \mathbb{R}^{d \times n}$ is given by $\DISC(M) := \min_{\bx \in \{-1,1\}^n} \|M\bx\|_\infty$. An outstanding conjecture, attributed to Koml\'os, stipulates that $\DISC(M) = O(1)$, whenever $M$ is a Koml\'os matrix, that is, whenever every column of $M$ lies within the unit sphere. Our main result asserts that $\DISC(M + R/\sqrt{d}) = O(d^{-1/2})$ holds asymptotically almost surely, whenever $M \in \mathbb{R}^{d \times n}$ is Koml\'os, $R \in \mathbb{R}^{d \times n}$ is a Rademacher random matrix, $d = \omega(1)$, and $n = \omega(d \log d)$. The factor $d^{-1/2}$ normalising $R$ is essentially best possible and  the dependency between $n$ and $d$ is asymptotically best possible. Our main source of inspiration is a result by Bansal, Jiang, Meka, Singla, and Sinha (ICALP 2022). They obtained an assertion similar to the one above in the case that the smoothing matrix is Gaussian. They asked whether their result can be attained with the optimal dependency $n = \omega(d \log d)$ in the case of Bernoulli random noise or any other types of discretely distributed noise; the latter types being more conducive for Smoothed Analysis in other discrepancy settings such as the Beck-Fiala problem. 
For Bernoulli noise, their method works if $n = \omega(d^2)$. 
In the case of Rademacher noise, we answer the question posed by Bansal, Jiang, Meka, Singla, and Sinha. Our proof builds upon their approach in a strong way and provides a discrete version of the latter. Breaking the $n = \omega(d^2)$ barrier and reaching the optimal dependency $n = \omega(d \log d)$ for Rademacher noise requires additional ideas expressed through a rather meticulous counting argument, incurred by the need to maintain a high level of precision all throughout the discretisation process. 



%
\end{abstract}

\section{Introduction}

The {\em discrepancy} of a matrix  $M \in \mathbb{R}^{d \times n}$ is given by $\DISC(M) := \min_{\bx \in \{-1,1\}^n} \|M\bx\|_\infty$. A celebrated result in this venue is the so-called ``six standard deviations" result, put forth by Spencer~\cite{Spencer85}, asserting that  if $\|M\|_\infty \leq 1$ and $d = n$, then $\DISC(M)\leq 6\sqrt{n}$. More generally, if $d \geq n$, then 
$\DISC(M) = O\left(\sqrt{n \log (2d/n)}\right)$ is known to hold~\cite{Bansal10,HLR17,LM15,Roth}. Spencer's result is essentially tight as $n \times n$ matrices $M$ satisfying $\DISC(M) = \Omega(\sqrt{n})$ are known to exist~\cite{CNN11}. 

An outstanding conjecture in Discrepancy Theory, attributed to Koml\'os, stipulates that $\DISC(M) = O(1)$ holds, whenever $M \in \mathbb{R}^{d \times n}$ has each of its columns $\bv$ satisfying $\|\bv\|_2 \leq 1$; we refer to the latter as a {\em Koml\'os matrix}\footnote{Koml\'os' restriction on the matrix is more stringent than that of Spencer.}. {\sl Dimension-free} (i.e., constant) bounds on the discrepancy of matrices are of special interest as it is $NP$-hard to distinguish between matrices having constant discrepancy and those having $\Omega(\sqrt{n})$ discrepancy~\cite{CNN11}.

Given a hypergraph $H$, taking $M = M_H$ to be its $e(H) \times v(H)$ incidence matrix retrieves the well-known (see, e.g.,~\cite{BS95,CH00}) notion of {\em combinatorial discrepancy}, given by 
$$
\DISC(H) := \min_{\chi} \max_{e \in E(H)} \left| \sum_{v \in e} \chi(v)\right|,
$$  
where the minimisation ranges over all mappings $\chi: V(H) \to \{-1,1\}$. Beck and Fiala~\cite{BF81} proved that if $H$ has the property that each of its vertices lies in at most $t$ edges, i.e., each column $\bv$ of $M_H$ satisfies $\|\bv\|_2 \leq \sqrt{t}$, then $\DISC(H) \leq 2t-1$, and conjectured that $\DISC(H) = O(\sqrt{t})$ holds in this case. Up to the $\sqrt{t}$-scaling, the Beck-Fiala conjecture is a special case of the Koml\'os conjecture. 

The best known upper bounds for the conjectures put forth by Koml\'os and by Beck-Fiala are $O\left(\sqrt{\log n} \right)$ and $O\left(\sqrt{t \log n} \right)$, respectively, both obtained by Banaszczyk~\cite{B98} in 1998\footnote{For the Beck-Fiala conjecture, see also~\cite{Bukh} for the currently best bound which is independent of $n$.}. Despite this partial progress, it seems that these two conjectures are out of reach of current techniques; consequently, the investigation of these conjectures in more hospitable settings, so to speak, is well-justified.  

One line of research that has attracted much attention of late calls for the determination of $\DISC(M)$ whenever $M$ is a random matrix; in this line of research one is interested in the so-called {\sl average-case} discrepancy or the discrepancy of {\sl typical} matrices, where `typical' depends on the specific distribution chosen for $M$.  In this realm, we further distinguish between two strands of study; the first pertains to gaussian matrices\footnote{\label{foot:gaussian-mat}Matrices with each entry an i.i.d. copy of $\mathcal{N}(\mu,\sigma^2)$; if $\mu = 0$ and $\sigma =1$, then the matrix is called a {\em standard} gaussian matrix.} and the second deals with discrete random matrices. 

For standard gaussian matrices $M \in \mathbb{R}^{d \times n}$, the estimate $\DISC(M) = \Theta\left(2^{-n/d}\sqrt{n}\right)$ holds asymptotically almost surely (a.a.s. hereafter) for a wide range of values of $d$ and $n$; in particular $\DISC(M) = O(1)$ holds as soon as $n \geq C d \log d$, where $C > 0$ is an appropriate constant. 
The case $d = O(1)$ of the above equality was settled by Costello~\cite{Costello}. Meka, Rigollet, and Turner~\cite{MRT20} extended the result of Costello by allowing $\omega(1) = d =o(n)$. In fact, their result accommodates any (matrix entry) distribution whose density function $f$ is symmetric, has a fourth moment, and is square-integrable. The regime $d = \Theta(n)$ was studied in~\cite{ALS,APZ,CV14,PX}. 

Proceeding to discrete random matrices, given $d \geq n \geq t$, Ezra and Lovett~\cite{EL19} proved that $\DISC(M) = O\left(\sqrt{t \log t}\right)$ holds with probability at least $1 - \exp(-\Omega(t))$, whenever each column of $M$ is sampled independently and uniformly at random from all $0/1$-vectors containing precisely $t$ non-zero entries. They also proved that $\DISC(M) =O(1)$ holds a.a.s. provided that $d \geq t$ and $n \gg d^t$. 
For Bernoulli matrices\footnote{Each entry is an independent copy of $\mathrm{Ber}(p)$ for $p := p(n,d)$.} $M \in \mathbb{R}^{d \times n}$, Altschuler and Niles-Weed~\cite{ANW22} proved that $\DISC(M) \leq 1$ holds a.a.s.  for any $p:= p(n)$, whenever $n \geq C d \log d$, where $C>0$ is an absolute constant\footnote{Discrepancy of Poisson matrices is also studied in~\cite{ANW22}; Bernoulli matrices are also studied in~\cite{HR19,P18}.}; their result is tight in terms of the lower bound on $n$. Moreover, their bound on the discrepancy is also best possible as any binary matrix that has a row with an odd number of 1's has discrepancy at least one.  


Given a {\em seed} matrix $M \in \mathbb{R}^{d \times n}$ as well as a distribution $\R_{d \times n}$, set over $\mathbb{R}^{d \times n}$, we refer to the (random) matrix $M+R$ with $R \sim \R_{d \times n}$ as a {\em random perturbation} of $M$. Following the aforementioned results pertaining to the discrepancy of truly random matrices, the study of the discrepancy of randomly perturbed ones is the next natural step. The study of the effect of random {\sl noise} is widespread in Mathematics and Computer Science. Spielman and Teng~\cite{ST09} coined the term {\sl smoothed analysis} to indicate the analysis of algorithms executed on randomly perturbed inputs. In high dimensional probability (see, e.g.,~\cite{Vershynin}), the study of randomly perturbed matrices dates back to the works of Tao and Vu~\cite{TV07,TV08,TV10}. In combinatorics, the study of randomly perturbed (hyper)graphs has witnessed a burst of activity in recent years; see, e.g.,~\cite{ADHLlarge,ADHLsmall,AHhamilton,AHK22b,AHK22a,AHTrees,AHS23,AHP,AHP22,BTW17,BHKM18,BFKM04,BFM03,BHKMPP18,BMPP18,DRRS18,HZ18,KKS16,KKS17,KST,MM18}. 

The main source of inspiration for our work is a result by Bansal, Jiang, Meka, Singla, and Sinha~\cite{BJMSS} who established the first ever perturbed/smoothed version of the Koml\'os conjecture. They proved that $\DISC(M+R) \leq \frac{1}{\mathrm{poly}(d)}$ holds a.a.s. whenever $M \in \mathbb{R}^{d \times n}$ is a Koml\'os matrix, $R \in \mathbb{R}^{d \times n}$ is a matrix whose entries are i.i.d. copies of $\mathcal{N}(0,\sigma^2/d)$ and $n = \omega(d \log d) \cdot \sigma^{-4/3}$. In~\cite[Section~3]{BJMSS}, Bansal, Jiang, Meka, Singla, and Sinha ask whether analogous results can be proved if instead of gaussian noise one uses discrete noise such as the Bernoulli distribution or some other natural discrete distribution. The interest in discrete noise models is reasoned in~\cite{BJMSS} as being more conducive for Smoothed Analysis in other discrepancy settings such as the Beck-Fiala problem. Bansal, Jiang, Meka, Singla, and Sinha~\cite[Section~3]{BJMSS} note that they are able to prove the required results if the smoothing noise has the Bernoulli distribution and $n = \omega(d^2)$. Attaining the optimal dependency $n = \omega(d \log d)$ for discrete noise models is then of interest and seems to require additional tools.


\subsection{Our contribution}\label{sec:contribution}




A random variable $X$ is said to be {\em Rademacher} if $X$ assumes the values $-1$ and $1$, each with probability $1/2$. A matrix $R \in \mathbb{R}^{d \times n}$ is said to form a {\em Rademacher matrix} if its entries are independent Rademacher random variables. Our main result reads as follows. 

\begin{theorem}\label{thm:main}
Let $d = \omega(1)$ and $n = \omega(d \log d)$ be integers. Then, $\DISC(M +R/\sqrt{d}) \leq 8d^{-1/2}$ holds a.a.s. whenever $M \in \mathbb{R}^{d \times n}$ is a Koml\'os matrix and $R \in \mathbb{R}^{d \times n}$ is a Rademacher matrix. 
\end{theorem}

\noindent
This resolves the aforementioned question of Bansal, Jiang, Meka, Singla, and Sinha~\cite[Section~3]{BJMSS} in the case that the smoothing noise has the Rademacher distribution. In the case that the noise has the Bernoulli distribution, the aforementioned dependency $n = \omega(d^2)$, stated in~\cite[Section~3]{BJMSS}, is the state of the art. 




\begin{remark}\label{rm:normalisation}{\bf Normalisation factor - lower bound.} 
{\em In Theorem~\ref{thm:main}, the Rademacher matrix $R$ is normalised by a $\sqrt{d}$ factor. This normalisation factor is warranted. Indeed, requiring that $\|\bv\|_2 \leq 1$ holds for every column $\bv$ of the random perturbation is a natural constraint to impose, for such a restriction guarantees that the columns of the perturbation do not dominate the columns of $M$. Writing $k := k(d)$ to denote the normalisation factor and letting $\bv$ be any column vector of $R/k$, we see that $1 \geq \|\bv\|_2^2 = \sum_{i=1}^d \frac{1}{k^2} = \frac{d}{k^2}$ implies $k \geq \sqrt{d}$. }
\end{remark}


\begin{remark}\label{rm:normalisation-2}{\bf Normalisation factor - upper bound.} 
{\em Let $k$ be as defined in Remark~\ref{rm:normalisation}. Enlarging $k$ is of interest as this reduces the dominance of the random perturbation further, allowing one to come ever closer to Koml\'os' conjecture. Alas, in the setting of Theorem~\ref{thm:main}, there is an upper bound on the normalisation factor $k$. To see this, note that given $k$ and a discrepancy bound $\Delta$, the stipulation that $\DISC(M+R/k) \leq \Delta$ is equivalent to requiring the existence of a vector $\bx \in \{-1,1\}^n$ for which 
\begin{equation}\label{eq:rad-sum}
(R\bx)_i \in \left[- k (M\bx)_i - k \Delta, - k (M\bx)_i + k \Delta \right]
\end{equation}
holds for every $i \in [d]$. Given $\bx \in \{-1, 1\}^n$ and $i \in [d]$, the term $(R \bx)_i$ has the same distribution as the sum $\sum_{i=1}^n r_i$, whose summands are independent Rademacher random variables. As such, $(R\bx)_i \in [- \omega(\sqrt{n}), \omega(\sqrt{n})]$ asymptotically almost surely. Consequently, a prerequisite for~\eqref{eq:rad-sum} holding a.a.s. is that  
$$
\left[- k (M\bx)_i - k \Delta, - k (M\bx)_i + k \Delta \right] \cap [- \omega(\sqrt{n}), \omega(\sqrt{n})] \neq \emptyset
$$ 
holds for every $i \in [d]$. 
Assuming that $\Delta$ is relatively small (as one naturally aims to have), the latter amounts to essentially requiring that $k \leq \sqrt{n}\|M\bx\|_\infty^{-1}$. The smaller the value of $\|M\bx\|_\infty$ we obtain, the less restrictive on $k$ this inequality becomes. In our current state of knowledge, the best we can ensure are vectors $\bx \in \{-1,1\}^n$ for which $\|M\bx\|_\infty = O(\sqrt{\log d})$ (see Lemma~\ref{lem:concentration}(ii)). Such a vector then yields the upper bound $k = O\left(\sqrt{n/\log d}\right)$. It follows that for $n = \omega(d \log d)$ (as in the premise of Theorem~\ref{thm:main}), taking $k$ to be roughly $\sqrt{d}$ is essentially best possible.
}
\end{remark}

\begin{remark}{\bf Dependence between $\boldsymbol{n}$ and $\boldsymbol{d}$.} 
{\em The requirement $n = \omega(d \log d)$ appearing in Theorem~\ref{thm:main} is asymptotically best possible. To see this, take $M$ to be the zero matrix (which is Koml\'os) and note that it suffices to prove that $\DISC(R) = O(1)$ mandates $n = \omega(d \log d)$.  To this end, fix an arbitrary vector $\bx \in \{-1, 1\}^n$. Given any constant $C>0$ and a row $\br$ of the Rademacher matrix $R \in \{-1, 1\}^{d \times n}$, we may write 
$$
\Pr\left[ |\langle \br,\bx \rangle| \leq C \right] = \sum_{k=-C/2}^{C/2} \frac{\binom{n}{n/2+k}}{2^n} = O \left(n^{- 1/2} \right),
$$
where the first equality holds since $\langle \br, \bx \rangle$ has the same distribution as a sum of i.i.d. Rademacher random variables, and the second equality is  supported by Proposition~\ref{prop:Spencer} below. It follows by the independence of the entries of $R$ that $\Pr[\|R \bx\|_\infty \leq C] = O \left(n^{-d/2} \right)$. Call a vector $\bx \in \{-1, 1\}^n$ \emph{good} if $\|R\bx\|_\infty \leq C$. Then
$$
\Ex[\text{number of good vectors}] = O \left( 2^n n^{-d/2} \right) = O \left(\exp(n - d\log n/2) \right).
$$ 
If $n = o(d \log d)$, then the above expectation vanishes and it thus follows by Markov's inequality that a.a.s. no good vectors exist. 
}
\end{remark}

\subsection{Our approach} 

A natural approach towards proving Theorem~\ref{thm:main} is to use the CLT in conjunction with the work of Bansal, Jiang, Meka, Singla, and Sinha~\cite{BJMSS} for the case of Gaussian noise. Our attempts to employ this approach stalled at the sub-optimal dependency $n = \omega(d^2)$; the same bound attained in~\cite{BJMSS} for Bernoulli noise. To break the $n = \omega(d^2)$ barrier and reach the optimal dependency $n = \omega(d \log d)$, we build upon the framework of Bansal, Jiang, Meka, Singla, and Sinha~\cite{BJMSS} in a strong way. In that, we provide a discretisation of their argument without appealing to the CLT. This calls for  a meticulous counting argument whose execution requires new ideas and various adaptations of the argument of~\cite{BJMSS}. 

In the Gaussian case, Bansal, Jiang, Meka, Singla, and Sinha~\cite{BJMSS} equip $\{-1,1\}^n$ with an appropriate distribution (see~\cite[Lemma~2.1]{BJMSS} and also Lemma~\ref{lem:concentration}) from which a vector $\bx \in \{-1,1\}^n$ satisfying $\|(M + R)\bx\|_\infty \leq 1/\mathrm{poly}(d)$ is identified, where here $R$ is a (scaled) conformal Gaussian matrix. Writing $S$ for the number of vectors $\bx$ satisfying the above bound, they employ the Paley-Zygmund inequality (see~\eqref{eq:Paley-Zyg}) in order to prove that $\Pr[S >0] \geq 1 - o(1)$. The bulk of the argument consists of establishing that $\Ex_R[S^2] \leq (1+o(1))\Ex_R[S]^2$. This is also the core of our proof of Theorem~\ref{thm:main}. 

Executing the above approach with Rademacher noise, summons various challenges. Roughly put, in~\cite{BJMSS} one encounters the need to estimate the probability that a gaussian vector lies, say, within some rectangular region. In the gaussian case, this can be handled through a nontrivial estimation (as seen in~\cite{BJMSS}) of a certain integral involving the probability density function of an adequate multi-dimensional gaussian distribution. A similar situation in the Rademacher case, is significantly more involved and, in order to not lose track of the optimal dependency $n = \omega(d \log d)$, requires a careful counting argument. The need to take on such estimations abound throughout the proof. Moreover, such changes mandate various subtle adjustments to earlier parts of the argument. For example, we require an adaption of the original aforementioned distribution defined in~\cite[Lemma~2.1]{BJMSS} (see Lemma~\ref{lem:concentrationS}).





\section{Preliminaries}

This section is divided into two subsections, both containing auxiliary results facilitating our proof of Theorem~\ref{thm:main}.

\begin{remark} \label{rem::BCsymmetry}
{\em Throughout this section we encounter binomial coefficients of the form $\binom{n}{n/2+t}$, where $n \in \mathbb{N}$ is even and $t \in \mathbb{Z}$. Owing to the symmetry $\binom{n}{n/2+t} = \binom{n}{n/2-t}$, whenever it is convenient, we assume that $t \geq 0$.}
\end{remark}



\subsection{Key tools} \label{sec:Spencer}


Our overall goal in Theorem~\ref{thm:main} is to prove the existence of a vector $\bx \in \{-1, 1\}^n$ for which a.a.s. $\|(M + R/\sqrt{d}) \bx\|_\infty \leq 8 d^{-1/2}$ holds. In order to do so we follow the core innovative technique put forth by Bansal, Jiang, Meka, Singla, and Sinha~\cite{BJMSS} and sample the vectors of $\{-1, 1\}^n$ according to an adaptation of a distribution $\D := \D_n$, called the {\em truncated Gram-Schmidt distribution}, defined below in Lemma~\ref{lem:concentration}.

A real random variable $X$ is said to be $\alpha$-{\em subgaussian}\footnote{Subgaussian random variables admit several equivalent characterisations; see, e.g.,~\cite[Proposition~2.5.2]{Vershynin} for details.} if it satisfies $\Pr[|X| \geq t] \leq 2 \exp(-(t/\alpha)^2)$ for every $t > 0$. A random vector $\bx \in \mathbb{R}^n$ is said to be $\alpha$-{\em subgaussian} if $\la \bx,\by \ra$ is $\alpha$-subgaussian for every $\by \in \mathbb{S}^{n-1}$, see, e.g.,~\cite[Definition~3.4.1]{Vershynin}. The following is one of the main results of~\cite{HSSZ19}. 

\begin{theorem}\label{thm:Gram-Schmidt}{\em~\cite{HSSZ19}}
Let $\bv^{(1)},\ldots,\bv^{(n)} \in \mathbb{R}^m$ satisfy $\|\bv^{(i)}\|_2 \leq 1$ for every $i \in [n]$. Applying the Gram-Schmidt walk sampling algorithm\footnote{See~\cite{HSSZ19} for details.} over the given vectors outputs a random vector $\bx \in \{-1,1\}^n$ such that the vector $\sum_{i=1}^n \bx_i \bv^{(i)}$ is $1$-subgaussian. 
\end{theorem}
 
The distribution (implicitly) defined in Theorem~\ref{thm:Gram-Schmidt} is {\sl truncated} in~\cite{BJMSS} so as to produce the following distribution over the vectors in $\{-1,1\}^n$.

\begin{lemma} \label{lem:concentration} {\em~\cite[Lemma~2.1]{BJMSS}}
Let $M \in \mathbb{R}^{d \times n}$ be a Koml\'os matrix. Then, there exists a constant $C_{\ref{lem:concentration}} > 0$ as well as a distribution $\D := \D_
n$, set over the vectors in $\{-1,1\}^n$, such that the following three properties hold simultaneously.
\begin{enumerate}
\item [$(i)$] $\|M \bx\|_2 \in \left[r - \mathcal{T}, r + \mathcal{T} \right]$ holds for every $\bx \in \supp \D$, where $r = O(\sqrt{d})$ and $\mathcal{T} = d^{-C'}$ for some constant $C' > 1$.

\item [$(ii)$] $\|M \bx\|_{\infty} = O \left(\sqrt{\log d} \right)$ holds for every $\bx \in \supp \D$.

\item [$(iii)$] $\Pr_{\bx \sim \D} \left[ |\la \bx, \bu \ra | \geq t \right] \leq d^{C_{\ref{lem:concentration}}}\exp(-t^2/8)$ and $\Pr_{\bx \sim \D} \left[ |\la M\bx, \bv \ra | \geq t \right] \leq d^{C_{\ref{lem:concentration}}}\exp(-t^2/8)$
both hold whenever $\bu \in \mathbb{S}^{n-1}$, $\bv \in \mathbb{S}^{d-1}$, and $t > 0$.
\end{enumerate}
\end{lemma}

\begin{remark} \label{rem::infinityNorm}
{\em Part (ii) of Lemma~\ref{lem:concentration} is not stated in~\cite{BJMSS}. Its proof being essentially the same as the proof of Lemma~\ref{lem:concentration}(i) in~\cite{BJMSS} is thus omitted}.
\end{remark}

We require certain extensions of the distribution $\D$. Given a non-negative integer $k$ and an injective mapping $\varphi : \{-1,1\}^n \to \{-1,1\}^{n+k}$, a distribution $\cS$ over $\{-1,1\}^{n+k}$ is said to be a {\em deterministic $\varphi$-extension} of $\D_n$ if the latter can be obtained by first sampling a vector $\bx \sim \D_n$ and then applying $\phi$ to $\bx$. If the specific nature of the mapping $\varphi$ is inconsequential, then we simply say that $\cS$ forms a {\sl deterministic extension} of $\D_n$. The following is an adaptation of Lemma~\ref{lem:concentration}, applicable to $\cS$.

\begin{lemma} \label{lem:concentrationS}
Let $i \in \{1,2\}$ and let $M \in \mathbb{R}^{d \times n}$ be a Koml\'os matrix whose last $i$ columns form the zero vector $\bold{0}$. Let $\cS$ be a distribution over $\{-1,1\}^n$ which forms a deterministic extension of $\D_{n-i}$. Then, there exists a constant $C_{\ref{lem:concentrationS}} > 0$ such that the following three properties hold simultaneously.
\begin{enumerate}
\item [$(i)$] $\|M \bx\|_2 \in \left[r - \mathcal{T}, r + \mathcal{T} \right]$ holds for every $\bx \in \supp \cS$, where $r = O(\sqrt{d})$ and $\mathcal{T} = d^{-C'}$ for some constant $C' > 1$.

\item [$(ii)$] $\|M \bx\|_{\infty} = O \left(\sqrt{\log d} \right)$ holds for every $\bx \in \supp \cS$.

\item [$(iii)$] $\Pr_{\bx \sim \cS} \left[ |\la \bx, \bu \ra | \geq t \right] \leq d^{C_{\ref{lem:concentrationS}}} \exp(-t^2/9)$ and $\Pr_{\bx \sim \cS} \left[ |\la M\bx, \bv \ra | \geq t \right] \leq d^{C_{\ref{lem:concentrationS}}}\exp(-t^2/8)$
both hold whenever $\bu \in \mathbb{S}^{n-1}$, $\bv \in \mathbb{S}^{d-1}$, and $t > 0$.
\end{enumerate}
\end{lemma}

\begin{proof}
Since the last $i$ columns of $M$ are $\bold{0}$, parts (i) and (ii) and the fact that $\Pr_{\bx \sim \cS} \left[ |\la M\bx, \bv \ra | \geq t \right] \leq d^{C_{\ref{lem:concentrationS}}} \exp(-t^2/8)$ holds for every $\bv \in \mathbb{S}^{d-1}$ and every $t > 0$, are immediate corollaries of their counterparts in Lemma~\ref{lem:concentration}.

Fix $t > 0$ and an arbitrary vector $\bu \in \mathbb{S}^{n-1}$. For every vector $\bv \in \mathbb{R}^n$, let $\bv^{(i)} \in \mathbb{R}^{n-i}$ be the vector consisting of the first $n-i$ coordinates of $\bv$. Then 
\begin{align*}
\Pr_{\bx \sim \cS} \left[ |\la \bx, \bu \ra | \geq t \right] &\leq \Pr_{\bx^{(i)} \sim \D_{n-i}} \left[ \left| \left\la \bx^{(i)}, \bu^{(i)} \right\ra \right| \geq t - 2 \right] \\
&= \Pr_{\bx^{(i)} \sim \D_{n-i}} \left[ \left| \left\la \bx^{(i)}, \|\bu^{(i)}\|_2^{-1} \bu^{(i)} \right\ra \right| \geq (t - 2) \|\bu^{(i)}\|_2^{-1} \right] \\
&\leq d^{C_{\ref{lem:concentration}} + 1} \exp \left(- (t-2)^2/8 \right) \\
&\leq d^{C_{\ref{lem:concentrationS}}} \exp \left(- t^2/9 \right),
\end{align*}
where the first inequality holds since $\bx \in \{-1,1\}^n$ and $\bu \in \mathbb{S}^{n-1}$, and the second inequality holds by Lemma~\ref{lem:concentration}(iii) and since $\|\bu^{(i)}\|_2 \leq \|\bu\|_2 = 1$ (multiplying by $d$ is used to circumvent the case $t \leq 2$; additionally, if $\|\bu^{(i)}\|_2 = 0$, then the claim is trivial).
\end{proof}

Given two vectors $\bx,\by \in \{-1,1\}^n$, let $\varepsilon := \varepsilon(\bx, \by) = \frac{\langle \bx, \by \rangle}{n}$, let $\Diff(\bx,\by) = \{i \in [n] : \bx_i \neq \by_i\}$, and let $\alpha := \alpha(\bx,\by) = 1 - \frac{|\Diff(\bx,\by)|}{n}$. Note that $|\Diff(\bx, \by)|$ is the Hamming distance between $\bx$ and $\by$, and  $\alpha(\bx,\by) n$ is the number of indices over which these two vectors coincide. The following result presents simple but useful relations between these parameters.

\begin{claim}
\label{claim:abc}
    Let $\bx,\by \in \{-1,1\}^n$ and let $\varepsilon := \varepsilon(\bx, \by)$ and $\alpha := \alpha(\bx,\by)$ be as above. Then
    \begin{enumerate}
        \item [$(a)$] $\alpha=\frac{1+\epsilon}{2}$;
        \item [$(b)$] $\frac{1}{\alpha(1-\alpha)} = \frac{4}{1-\epsilon^2} \leq 4 \exp \left(2\epsilon^2 \right)$, where the inequality holds provided that $|\varepsilon| \leq 1/2$.
    \end{enumerate}
\end{claim}

\begin{proof}
    Starting with (a), note that
    \begin{align*}
        \langle \bx, \by \rangle &=
        \sum_{i\in [n] \backslash \Diff(\bx,\by)} \bx_i \by_i
        +\sum_{i\in\Diff(\bx,\by)} \bx_i \by_i
        = (n - |\Diff(\bx,\by)|) - |\Diff(\bx,\by)| = n - 2 |\Diff(\bx,\by)|.
    \end{align*}
    It thus follows that
    $\frac{|\Diff(\bx,\by)|}{n}=\frac{n-\langle \bx,\by \rangle}{2n}= \frac{1-\epsilon}{2}$.
    Hence
    $$
        \alpha = 1 - \frac{|\Diff(\bx,\by)|}{n} =1 - \frac{1-\epsilon}{2}
        = \frac{1 + \epsilon}{2}.
    $$
    Next, we prove (b). Using (a) we obtain
    \begin{align*}
        \frac{1}{\alpha(1-\alpha)} &=
        \frac{1}{\frac{1+\epsilon}{2} \cdot \frac{1-\epsilon}{2}}
        =
        \frac{4}{(1+\epsilon)(1-\epsilon)}
        =
        \frac{4}{1-\epsilon^2}.
    \end{align*}
    Assume now that $|\varepsilon| \leq 1/2$. Since $1-x \geq \exp(-2x)$ holds whenever $0 \leq x \leq 1/2$, it follows that
    $$
        \frac{4}{1-\epsilon^2} \leq \frac{4}{\exp\left(-2\epsilon^2\right)} = 4\exp\left(2\epsilon^2\right).
    $$
\end{proof}

A key tool in our approach is the following approximation result for binomial coefficients $\binom{n}{k}$, where $k$ is ``close'' to $n/2$. 

\begin{proposition}\label{prop:Spencer}
Let $n$ be a sufficiently large even integer and let $t \in \mathbb{Z}$ be such that $|t| = o(n)$ and $\frac{n+t}{2} \in \mathbb{Z}$. Then, 
\begin{equation}\label{eq:Spencer}
 \binom{n}{\frac{n+t}{2}} = (1+o_n(1))\sqrt{\frac{2}{\pi n}} \cdot 2^n \exp{\left(- \frac{t^2}{2n} + \Theta\left(\frac{t^3}{n^2}\right) + o \left(\frac{t}{n} \right)\right)}.
\end{equation}
\end{proposition}

\begin{remark}
{\em Up to small modifications, Proposition~\ref{prop:Spencer} and its proof can be found in~\cite[Section~5.4]{Spencer}; we include the proposition and its proof here as these modifications are important for our purposes}. 
\end{remark}

\begin{proofof}{Proposition~\ref{prop:Spencer}}
Let
$$
Q = \binom{n}{\frac{n+t}{2}}/\binom{n}{n/2} = \frac{(n/2)! (n/2)!}{\left(\frac{n+t}{2}\right)! \left(\frac{n-t}{2}\right)!} = \prod_{j=1}^{t/2} \frac{n/2 - j+1}{n/2+j}.
$$
Therefore
\begin{equation}\label{eq:logQ-1}
\log Q = \sum_{j=1}^{t/2} \log \left(1 - \frac{4j-2}{n+2j}\right) = \sum_{j=1}^{t/2} \left[- \frac{4j-2}{n+2j} + \Theta\left(\frac{j^2}{n^2}\right) \right],
\end{equation}
where for the last equality we use the expansion $\log (1-x) = -x + \Theta(x^2)$, holding whenever $x \in (0,1)$. 
Substituting the identity 
$$
\frac{4j-2}{n+2j} = \frac{4j}{n} - \frac{8j^2}{n(n+2j)} - \frac{2}{n+2j} = \frac{4j}{n} - \frac{2}{n+2j} + \Theta \left(\frac{j^2}{n^2} \right)
$$
into~\eqref{eq:logQ-1} yields
\begin{equation}\label{eq:logQ-2}
\log Q = -\sum_{j=1}^{t/2} \frac{4j}{n} +\sum_{j=1}^{t/2} \frac{2}{n+2j}+ \sum_{j=1}^{t/2}\Theta\left( \frac{j^2}{n^2}\right) 
= -\frac{t}{n} - \frac{t^2}{2n} +  \sum_{j=1}^{t/2} \frac{2}{n+2j}+ \Theta(t^3/n^2),
\end{equation}
where for the last equality we employ the identity $\sum_{i=1}^k i = k(k+1)/2$ and the estimate $\sum_{i=1}^k i^2 = \Theta(k^3)$. 

The sum appearing on the right hand side of~\eqref{eq:logQ-2} satisfies
\begin{equation*} 
\frac{t}{n+t} = \sum_{j=1}^{t/2} \frac{2}{n+t} \leq \sum_{j=1}^{t/2} \frac{2}{n+2j} \leq \sum_{j=1}^{t/2} \frac{2}{n} = \frac{t}{n}. 
\end{equation*}
Since $t = o(n)$, it follows that
\begin{equation} \label{eq::essentiallyt/n}
\sum_{j=1}^{t/2} \frac{2}{n+2j} = (1 + o(1)) t/n = t/n + o(t/n).
\end{equation}

Combining~\eqref{eq:logQ-2} and~\eqref{eq::essentiallyt/n} then implies that
$$
\log Q = - \frac{t^2}{2n} + \Theta\left(\frac{t^3}{n^2}\right) + o \left(\frac{t}{n} \right).
$$
The claim follows since 
$$
\binom{n}{n/2} = \left(1+o_n(1)\right) \sqrt{\frac{2}{\pi n}}\cdot 2^n
$$
holds by a straightforward application of Stirling's approximation\footnote{Use $\sqrt{2 \pi n} \left(\frac{n}{e}\right)^ne^{1/(12n+1)} \leq  n! < \sqrt{2 \pi n} \left(\frac{n}{e}\right)^ne^{1/12n}$.}.
\end{proofof}

\begin{lemma} \label{lem::choosesX}
Let $i \in \{1,2\}$ and let $M \in \mathbb{R}^{d \times n}$ be a Koml\'os matrix whose last $i$ columns are $\bold{0}$. Let $\cS$ be a distribution over $\{-1,1\}^n$ which forms a deterministic extension of $\D_{n-i}$. Then, for every $\bx \in \supp \cS$, there exists a vector $s^{\bx} = (s^{\bx}_1, \ldots s^{\bx}_d)$ which satisfies the following three properties.
\begin{enumerate}
\item [$(1)$] $s^{\bx}_i \in [-4, 4]$ for every $i \in [d]$;

\item [$(2)$] $- \sqrt{d} (M \bx)_i + s^{\bx}_i \equiv n \pmod 4$ for every $i \in [d]$;

\item [$(3)$] $\left|\sum_{i=1}^d s^{\bx}_i (M \bx)_i \right| = O(\sqrt{\log d})$.
\end{enumerate}
\end{lemma}

\begin{proof}
Note first that for every positive integer $n$ and any real number $a$ there exists a unique real number $a' \in [0, 4)$ such that $a + a' \equiv n \pmod 4$; let $f_n : \mathbb{R} \to \mathbb{R}$ be defined by $f(a) = a'$. Given any $\bx \in \cS$ we determine the coordinates of $s^{\bx}$ sequentially. Set $s^{\bx}_1 = f_n(- \sqrt{d} (M \bx)_1)$. Suppose we have already determined $s^{\bx}_1, \ldots s^{\bx}_j$ for some $j \in [d-1]$ and now aim to choose $s^{\bx}_{j+1}$. If $(M \bx)_{j+1} \cdot \sum_{i=1}^j s^{\bx}_i (M \bx)_i < 0$, then we set $s^{\bx}_{j+1} = f_n(- \sqrt{d} (M \bx)_{j+1})$; in all other cases we set $s^{\bx}_{j+1} = f_n(- \sqrt{d} (M \bx)_{j+1}) - 4$. Observe that Properties (1) and (2) follow immediately from the definition of $f_n$ and from our process of choosing $s^{\bx}_1, \ldots s^{\bx}_d$. Similarly, Property (3) follows from our process of choosing $s^{\bx}_1, \ldots s^{\bx}_d$ and by Lemma~\ref{lem:concentrationS}(ii).    
\end{proof}

\subsection{Rudimentary probabilistic estimations}\label{sec:col-cnt}

The main results of this section are Lemmas~\ref{lem:single-even} and~\ref{lem:joint-even} stated below. The proofs of these being rudimentary (yet crucial to subsequent arguments) are thus delegated to Appendix~\ref{app:rud}. Roughly put, these two lemmas deal with determining the probabilities of events of the form $\la \br, \bx \ra = 2k$, where $\br$ is a Rademacher vector, $\bx \in \{-1,1\}^n$, $n \in \mathbb{N}$ is even, and $k \in \mathbb{Z}$; we refer to such probabilities as {\sl core probabilities}.  The focus on the inner product being even is owing to the fact that  $\sum_{i=1}^n \by_i = \#_1(\by) - \#_{-1}(\by)$ holds for any vector $\by \in \{-1,1\}^n$. Since $n$ is even, there exists an integer $y$ such that $\#_1(\by) = n/2 + y$ leading to $\sum_{i=1}^n \by_i = n/2 +y - (n/2 -y) = 2y$. The following is then implied. 

\begin{observation}\label{obs:sum-2t}
Let $n$ be a positive even integer and let $k \in \mathbb{Z}$. Then, 
\begin{equation}\label{eq:sum-2t}
|S_k| = \binom{n}{n/2+k},
\end{equation}
where $S_k := \Big\{\bv \in \{-1,1\}^n: \sum_{i=1}^n \bv_i = 2k\Big\}$.
\end{observation} 

\medskip

Let 
$$
\mathcal{E}_n = \Big\{\bv \in \{-1,1\}^n: \#_1(\bv) \equiv 0 \; \pmod 2\Big\}
$$ 
denote the set of so-called {\em even} members of $\{-1,1\}^n$. The first main result of this section reads as follows.  

\begin{lemma}\label{lem:single-even}
Let $n \in \mathbb{N}$ be even, let $\br$ be a vector sampled uniformly at random from $\mathcal{E}_n$, let $\bx \in \{-1,1\}^n$, and let $k \in \mathbb{Z}$ be such that $\Pr[\la \br, \bx \ra = 2k] > 0$. Then,
\begin{equation}\label{eq:single-even}
\Pr\Big[\la \br, \bx \ra = 2k \Big] = \frac{1}{2^{n-1}} \binom{n}{n/2+k}.
\end{equation}
\end{lemma}

The second main result of this section reads as follows. 

\begin{lemma}\label{lem:joint-even}
Let $n \in \mathbb{N}$ be even, let $\br$ be a vector sampled uniformly at random from $\mathcal{E}_n$, let $\bx,\by \in \{-1,1\}^n$ satisfying $\#_1(\bx) \equiv \#_1(\by) \pmod{2}$ be given, and let $\alpha = \alpha(\bx,\by)$ be as in Claim~\ref{claim:abc}. Then, for any pair of integers $k_{\bx}$ and $k_{\by}$ satisfying $\Pr \left[\la \br,\bx \ra = 2 k_{\bx}, \la \br,\by \ra =2 k_{\by} \right] > 0$, the equality  
\begin{align}
\Pr \Big[\la \br,\bx \ra = 2 k_{\bx}, \la \br,\by \ra =2 k_{\by} \Big] &= \frac{1}{2^{n-1}}\binom{\alpha n}{\frac{\alpha n + k_{\bx} + k_{\by} }{2}}\binom{(1-\alpha)n}{\frac{(1-\alpha) n + k_{\bx} - k_{\by}}{2}}\label{eq:joint-prob-1}
\end{align}
holds.
\end{lemma}

\section{Proof of the main result - Theorem~\ref{thm:main}} \label{sec:main-proof}

Using the fact that $\|M\|_\infty \leq 1$ holds whenever $M$ is Koml\'os (for indeed $\|\bv\|_\infty \leq \|\bv\|_2 \leq 1$ holds for every column $\bv$ of $M$), we deduce Theorem~\ref{thm:main} from the following claim. 

\begin{claim}\label{clm:main}
Let $d = \omega(1)$ be an integer and let $n = \omega(d \log d)$ be an even integer. Let $\cS$ be a distribution over $\{-1,1\}^n$ which forms a deterministic extension of $\D_{n-i}$, for some $i \in \{1,2\}$, and such that $\#_1(\bx) \equiv 0 \pmod{2}$ holds for every vector $\bx \in \supp \cS$. Let $M \in \mathbb{R}^{d \times n}$ be a Koml\'os matrix whose last $i$ columns are $\bold{0}$, and let $R \in \mathbb{R}^{d \times n}$ be a Rademacher matrix such that $\#_1(\br) \equiv 0 \pmod{2}$ holds for every row $\br$ of $R$. Then, a.a.s. there exists a vector $\bx \in \supp \cS$ such that $\|(M+R/\sqrt{d})\bx \|_\infty \leq 4 d^{-1/2}$ holds.
\end{claim}

\medskip
\noindent
{\bf Claim~\ref{clm:main} implies Theorem~\ref{thm:main}:} 
Let $n$ and $M$ per the premise of Theorem~\ref{thm:main} be given. Set $M_1 := [\;M \mid \bold{0} \;] \in \mathbb{R}^{d \times (n+1)}$ and $M_2 := [\;M \mid \bold{0} \mid \bold{0} \;] \in \mathbb{R}^{d \times (n+2)}$, where $\bold{0}$ denotes the zero vector in $\mathbb{R}^d$; in particular, $M_1$ and $M_2$ are both Koml\'os. Let $R_1 \in \mathbb{R}^{d \times (n+1)}$ and $R_2 \in \mathbb{R}^{d \times (n+2)}$ be Rademacher matrices, each satisfying the row parity condition stated in Claim~\ref{clm:main}. 

Given $\bx \in \supp \D$, define $\bx^{(1)} := [\bx \mid \ell] 
	\in \{-1,1\}^{n+1}$ and $\bx^{(2)} := [\bx \mid \ell_1 \mid 
	\ell_2] 
	\in \{-1,1\}^{n+2}$, where 
	$$
	\ell := 
	\begin{cases}
	-1, & \#_1(\bx) \equiv 0 \pmod{2},\\
	\phantom{-}1, & \#_1(\bx) \equiv 1 \pmod{2},
	\end{cases}
	$$
and 
	$$
	(\ell_1,\ell_2) := 
	\begin{cases}
	(-1,-1), & \#_1(\bx) \equiv 0 \pmod{2},\\
    (-1, \phantom{-}1), & \#_1(\bx) \equiv 1 \pmod{2}.
	\end{cases}
	$$
It follows that $\#_1(\bx^{(1)})\equiv \#_1(\bx^{(2)}) \equiv 0 \pmod{2}$ holds for every $\bx \in \supp \D$. For $i \in \{1,2\}$, define $\mathcal{S}_i$ to be a distribution set over $\{-1,1\}^{n+i}$ obtained by first sampling a vector $\bx \in \{-1,1\}^n$ according to the distribution $\D$ and then performing the (injective) deterministic extension yielding $\bx^{(i)}$.

If $n$ is odd, then set $N:= M_1$, $\mathcal{S} = \mathcal{S}_1$, and $R:= R_1$; otherwise set $N:= M_2$, $\mathcal{S} = \mathcal{S}_2$, and $R = R_2$. Claim~\ref{clm:main} asserts that a.a.s. there exists a vector $\by \in \supp \cS$ for which $\|(N+R/\sqrt{d})\by\|_\infty \leq 4d^{-1/2}$ holds. Resampling the first entry of every row of $R$ allows for a conformal Rademacher matrix to be sampled uniformly at random at the price of increasing the discrepancy by at most $2 d^{-1/2}$ asymptotically almost surely. Expose $R$ and let $R'$ be the matrix obtained from $R$ by dropping its last column, if $n$ is odd, and its last two columns, if $n$ is even. In addition, let $\by' \in \{-1,1\}^n$ be the vector obtained from $\by$ by dropping its last entry, if $n$ is odd, and its last two entries, if $n$ is even. Note that, $\|(M+R'/\sqrt{d})\by'\|_\infty \leq 8d^{-1/2}$.   
\hfill\inQED

\medskip

The remainder of this section is devoted to the proof of Claim~\ref{clm:main}. For every $\bx \in \supp \cS$ and every $i \in [d]$, let $s^{\bx}_i$ be as in Lemma~\ref{lem::choosesX} and let $t^{\bx}_i = - \sqrt{d} (M \bx)_i + s^{\bx}_i$. Set $\Delta := 4d^{-1/2}$ and define the random variable
\begin{align*}
S := S(R) = \sum_{\bx \in \supp \cS} \mathbbm{1}\left\{R \bx = (t^{\bx}_1, \ldots, t^{\bx}_d) \right\} \cdot \Pr_{\by \sim \cS}[\by = \bx] = \Ex_{\bx \sim \cS} \left[\mathbbm{1}\left\{R \bx = (t^{\bx}_1, \ldots, t^{\bx}_d) \right\} \right]
\end{align*}
whose sole source of randomness is $R$. 
It suffices to prove that $S > 0$ holds asymptotically almost surely. Indeed, if the latter holds, then for {\sl almost every} Rademacher matrix $R$, there exists a vector $\bx \in \supp \cS$ for which 
$$
\mathbbm{1}\left\{R \bx = (t^{\bx}_1, \ldots, t^{\bx}_d) \right\} \cdot \Pr_{\by \sim \cS}[\by = \bx] > 0 
$$
holds. It then follows by our choice of $\Delta$ and by Lemma~\ref{lem::choosesX}(1) that for almost every Rademacher matrix $R$, there exists a vector $\bx \in \supp \cS$ for which the event 
$
\left\|\left(M+R/\sqrt{d}\right)\bx \right\|_\infty  \leq \Delta
$
occurs.

Establishing that $\Ex_R[S] >0$ (in Claim~\ref{clm:Ex>0} below) enables an appeal to the following consequence of the Paley-Zygmund inequality (see, e.g.,~\cite{Durrett})
\begin{equation}\label{eq:Paley-Zyg}
\Pr_R[S>0] \geq \frac{\Ex_R[S]^2}{\Ex_R[S^2]}.
\end{equation}
Hence, given that $\Ex_R[S] >0$ holds, it suffices to prove that 
\begin{equation}\label{eq:PZ-goal}
\Ex_R[S^2] \leq (1+o(1))\Ex_R[S]^2
\end{equation} 
in order to deduce that $\Pr_R[S>0] \geq 1-o(1)$.

\medskip

Prior to proving Claim~\ref{clm:Ex>0}, it will be useful to establish the following simple fact.
\begin{claim} \label{cl::Px>0}
 $\Pr_R \left[R \bx = (t^{\bx}_1, \ldots, t^{\bx}_d)  \right] > 0$ for every $\bx \in \supp \cS$.
\end{claim}

\begin{proof}
Fix an arbitrary vector $\bx \in \supp \cS$. Then,  
$$
\|M\bx\|_{\infty} = O(\sqrt{\log d}) < n/\sqrt{d}, 
$$
where the equality holds by Lemma~\ref{lem:concentrationS}(ii), and the inequality holds since $n$ is assumed to be sufficiently large with respect to $d$. It follows that $(M\bx)_i \in [-n/\sqrt{d}, n/\sqrt{d}]$ holds for every $i \in [d]$.

Since $n$ is even and $\#_1(\br) \equiv 0 \pmod{2}$ holds for every row $\br$ of $R$, it follows that for every $i \in [d]$ and every $k \in \{m \in [- n, n] : m \equiv n \pmod{4} \}$, there exists a vector $\br_i \in \mathcal{E}_n$ such that $\la \br_i, \bx \ra = k$. It then follows by Lemma~\ref{lem::choosesX}(2) that there exists a choice of $R$ with each of its rows satisfying the parity condition stated in Claim~\ref{clm:main} such that $(R\bx)_i = t^{\bx}_i$ holds for every $i \in [d]$; this concludes the proof of the claim.
\end{proof}


\begin{claim}\label{clm:Ex>0}
$\Ex_R[S] >0$.
\end{claim}

\begin{proof}
Note that 
\begin{align*}
\Ex_R[S] = \Ex_{\bx \sim \cS} \Ex_R \left[\mathbbm{1}\left\{R \bx = (t^{\bx}_1, \ldots, t^{\bx}_d) \right\} \right] = \Ex_{\bx \sim \cS} \Pr_R \left[R \bx = (t^{\bx}_1, \ldots, t^{\bx}_d) \right] > 0,
\end{align*}
where the above inequality holds by Claim~\ref{cl::Px>0}.
\end{proof}

Turning our attention to~\eqref{eq:PZ-goal}, note that
\begin{align*}
(\Ex_R[S])^2 = \left(\Ex_{\bx \sim \cS} \Pr_R \left[R \bx = (t^{\bx}_1, \ldots, t^{\bx}_d) \right]\right) \cdot \left(\Ex_{\by \sim \cS} \Pr_R \left[R \by = (t^{\by}_1, \ldots, t^{\by}_d) \right]\right) = \Ex_{\bx,\by \sim \cS} \left[P_\bx P_\by \right],
\end{align*}
where, for every $\bx \in \supp \cS$,
$$
P_\bx := \Pr_R \left[R \bx = (t^{\bx}_1, \ldots, t^{\bx}_d) \right].
$$
 
Similarly
\begin{align*}
\Ex_R[S^2] & = \Ex_R \left[\Ex_{\bx \sim \cS} \left[ \mathbbm{1} \left\{R \bx = (t^{\bx}_1, \ldots, t^{\bx}_d) \right\} \right] \cdot \Ex_{\by \sim \cS} \left[\mathbbm{1} \left\{ R \by = (t^{\by}_1, \ldots, t^{\by}_d) \right\} \right] \right]\\
& = \Ex_R \Ex_{\bx,\by\sim \cS} \left[ \mathbbm{1}\left\{R \bx = (t^{\bx}_1, \ldots, t^{\bx}_d) \right\} \cdot \mathbbm{1} \left\{R \by = (t^{\by}_1, \ldots, t^{\by}_d) \right\} \right] \\
& = \Ex_{\bx,\by \sim \cS} \left[ \Pr_R \left[R \bx = (t^{\bx}_1, \ldots, t^{\bx}_d), R \by = (t^{\by}_1, \ldots, t^{\by}_d) \right] \right] = \Ex_{\bx,\by \sim \cS} \left[ P_{\bx,\by} \right],
\end{align*}
where, for every $\bx, \by \in \supp \cS$, 
$$
P_{\bx,\by} := \Pr_R \left[R \bx = (t^{\bx}_1, \ldots, t^{\bx}_d), R \by = (t^{\by}_1, \ldots, t^{\by}_d) \right].
$$
The goal~\eqref{eq:PZ-goal} can then be rewritten as follows 
\begin{equation}\label{eq:new-goal}
\Ex_{\bx,\by \sim \cS} \left[ P_{\bx,\by} \right] \leq (1+o(1))\Ex_{\bx,\by \sim \cS} \left[ P_\bx P_\by \right].
\end{equation}

\medskip
We begin by considering the right hand side of~\eqref{eq:new-goal}. Our first result in this respect is an estimation of $P_{\bx}$.

\begin{lemma}
\label{lemma:Px1}
Suppose that $n = \omega(d)$. Then, for every $\bx \in \supp \cS$, it holds that 
    $$
    P_{\bx} = (1+o_d(1)) \left(\frac{8}{\pi n}\right)^{d/2} \cdot
    \exp\left( - \frac{1}{2n} \sum_{i=1}^{d} (t_i^\bx)^2 \right) \cdot \exp (\delta_{\bx}), 
    $$
    where $\delta_{\bx} = \left( O\left(n^{-2} d^{3/2} \log d \right) + o \left(n^{-1} \sqrt{d} \right) \right) |\la M\bx, \bold{1} \ra|$.
\end{lemma}

\begin{proof} 
Fix any $\bx \in \supp \cS$. It follows by the independence of the entries of $R$ that
\begin{align}
P_\bx &= \Pr_R\left[R\bx = (t_1^{\bx}, \ldots, t_d^{\bx}) \right] = \prod_{i=1}^d \Pr_R\left[(R\bx)_i = t_i^{\bx} \right]
= \prod_{i=1}^d \frac{1}{2^{n-1}}\binom{n}{\frac{n+t_i^{\bx}}{2}}\nonumber \\
&= \left(1+o_d(1)\right)\left(\frac{8}{\pi n} \right)^{d/2}\prod_{i=1}^d \exp \left(\Theta\left(\frac{(t_i^\bx)^3}{n^2}\right) \right) \exp \left(o\left(\frac{t_i^\bx}{n}\right) \right) \exp\left(- \frac{(t_i^\bx)^2}{2n}\right), \label{eq:Px-post-Stirling}
\end{align}
where the penultimate equality holds by Lemma~\ref{lem:single-even} and the last equality holds by Proposition~\ref{prop:Spencer} (note that $t_i^{\bx} = o(n)$ holds by Lemma~\ref{lem:concentrationS}(ii) and that $(1 + o_n(1))^d = 1 + o_d(1)$ holds by footnote 8 and since $n = \omega(d \log d)$).

In light of~\eqref{eq:Px-post-Stirling}, in order to complete the proof of the lemma, it suffices to prove that $\left|\sum_{i=1}^d t_i^{\bx} \right| = \sqrt{d} |\la M\bx, \bold{1} \ra| + o(n)$ and that $\left|\sum_{i=1}^d (t_i^{\bx})^3 \right| = O \left(d^{3/2} \log d |\la M\bx, \bold{1} \ra| \right) + o(n^2)$, which would in turn imply that
\begin{align*}
&  \prod_{i=1}^d \exp \left(\Theta\left(\frac{(t_i^\bx)^3}{n^2}\right) \right) \exp \left(o\left(\frac{t_i^\bx}{n}\right) \right) = \exp \left( \Theta \left(\frac{\sum_{i=1}^d (t_i^\bx)^3}{n^2} \right) + o \left(\frac{\sum_{i=1}^d t_i^\bx}{n} \right) \right) \\ 
&= (1+o_d(1)) \exp \left(\left( O\left(n^{-2} d^{3/2} \log d \right) + o \left(n^{-1} \sqrt{d} \right) \right) |\la M\bx, \bold{1} \ra| \right).
\end{align*}

Since $t_i^\bx = - \sqrt{d} (M\bx)_i + s^{\bx}_i$ holds for every $i \in [d]$, it follows that
\begin{align} \label{eq::sumtix}
   \left| \sum_{i=1}^d t_i^\bx \right| = \left| \sum_{i=1}^d - \sqrt{d} (M\bx)_i + s^{\bx}_i \right| \leq \sqrt{d} \left|\sum_{i=1}^d (M\bx)_i \right| + O(d) = \sqrt{d} |\la M \bx, \bold{1} \ra| + o(n),
\end{align}
where the last equality holds since $n = \omega (d)$ by the premise of the lemma.

\medskip
\noindent
Similarly, for every $i \in [d]$,
\begin{align*}
    (t_i^\bx)^3 = - d^{3/2} ((M\bx)_i)^3 + 3d s^{\bx}_i ((M\bx)_i)^2 - 3 \sqrt{d} (s^{\bx}_i)^2 (M\bx)_i + (s^{\bx}_i)^3.
\end{align*}
It thus follows that
\begin{align} \label{eq::sumtix3}
   \left| \sum_{i=1}^{d} (t_i^\bx)^3 \right| &\leq d^{3/2} \left| \sum_{i=1}^d ((M\bx)_i)^3 \right| + 3d \left| \sum_{i=1}^d s^{\bx}_i ((M\bx)_i)^2 \right| - 3 \sqrt{d} \left| \sum_{i=1}^d (s^{\bx}_i)^2  (M\bx)_i \right| + \left| \sum_{i=1}^d (s^{\bx}_i)^3 \right| \nonumber \\
   &\leq d^{3/2} \|M \bx\|_{\infty}^2 \left| \sum_{i=1}^d (M\bx)_i \right| + 12 d \|M \bx\|_2^2 + 48 \sqrt{d} \left| \sum_{i=1}^d ((M\bx)_i)^2 + 1 \right| + O(d) \nonumber \\
   &\leq d^{3/2} \log d |\la M \bx, \bold{1} \ra| + O(d^2) + 48 \sqrt{d} \|M \bx\|_2^2 + O(d^{3/2}) + O(d) \nonumber \\ 
   &\leq d^{3/2} \log d |\la M \bx, \bold{1} \ra| + o(n^2),
\end{align}
where the second inequality holds since $|s^{\bx}_i| \leq 4$ by Lemma~\ref{lem::choosesX}(1), the penultimate inequality holds by parts (i) and (ii) of Lemma~\ref{lem:concentrationS}, and the last inequality holds by Lemma~\ref{lem:concentrationS}(i) and since $n = \omega(d)$ by the premise of the lemma.
\end{proof}

Lemma~\ref{lemma:Px1} provides the following useful uniform estimation on the probabilities $P_{\bx}$.

\begin{lemma}
\label{lemma:Pxeqp}
Suppose that $n = \omega(d)$ and let $p = \left(\frac{8}{\pi n}\right)^{d/2} \, \cdot \exp\left(- r^2 \delta^2/2 \right)$,
where $r$ is as in Lemma~\ref{lem:concentration}(i) and $\delta = \sqrt{d/n}$.
Then, $P_{\bx} = (1+o_d(1)) p \cdot \exp \left(\delta_{\bx} \right)$ holds for every $\bx \in \supp \cS$, where $\delta_{\bx}$ is as in Lemma~\ref{lemma:Px1}.
\end{lemma}

\begin{proof} 
    Fix any $\bx \in \supp \cS$. In light of Lemma~\ref{lemma:Px1} it suffices to show that
    $$
        \frac{1}{2n}\sum_{i=1}^{d} (t_i^\bx)^2 = r^2 \delta^2/2 + o_d(1).
    $$
    Since $t_i^{\bx} = - \sqrt{d}(M\bx)_i + s_i^\bx$, where $s_i^\bx \in [-4,4]$ by Lemma~\ref{lem::choosesX}(1), holds for every $i \in [d]$, it follows that
    \begin{align} \label{eq::sumtix2}
        \sum_{i=1}^d (t_i^\bx)^2 &= d \sum_{i=1}^d ((M\bx)_i)^2
        - 2\sqrt{d} \sum_{i=1}^{d} s_i^\bx (M\bx)_i
        + \sum_{i=1}^{d} (s_i^\bx)^2
        \nonumber \\
        &= d\|M\bx\|^2_2 - 2\sqrt{d}\sum_{i=1}^d s_i^\bx (M\bx)_i + O(d),
    \end{align}
    where the last equality holds by Lemma~\ref{lem::choosesX}(1). For $r = O(\sqrt{d})$ and $\mathcal{T} = d^{-C'}$, with $C' >1$ some constant, $\| M\bx \|_2 \in \left[r - \mathcal{T}, r + \mathcal{T} \right]$ holds, by Lemma~\ref{lem:concentrationS}(i); this, in turn, implies that
    \begin{align*}
      r^2 - O(1) \leq (r-\mathcal{T})^2 \leq \|M\bx\|_2^2\leq (r+\mathcal{T})^2 \leq r^2 + O(1).
    \end{align*}
    Since, moreover, $n = \omega(d)$, $r = O(\sqrt{d})$, and $\delta = \sqrt{d/n}$, it follows that
    \begin{equation} \label{eq::2norm}
        \delta^2 \| M\bx \|_2^2/2 = \delta^2 r^2/2 \pm O \left(d n^{-1} \right) = \delta^2 r^2/2 + o_d(1).
    \end{equation}
    Combining~\eqref{eq::sumtix2} and~\eqref{eq::2norm}, leads to
    \begin{align*}
        \frac{1}{2n} \sum_{i=1}^d (t_i^\bx)^2 &= \frac{1}{2n} d \| M\bx \|_2^2 - \frac{1}{n} \sqrt{d} \sum_{i=1}^{d} s_i^\bx (M\bx)_i + O \left(n^{-1} d \right) \\
        &= \delta^2 r^2/2 - \frac{1}{n} \sqrt{d} \sum_{i=1}^{d} s_i^\bx (M\bx)_i + o_d(1).
    \end{align*}
    In order to conclude the proof of the lemma, it remains to prove that $\left|n^{-1} \sqrt{d}\sum_{i=1}^d s_i^\bx (M\bx)_i \right| = o_d(1)$. Indeed, it follows by Lemma~\ref{lem::choosesX}(3) that 
    \begin{align*}
        \left|n^{-1} \sqrt{d}\sum_{i=1}^d s_i^\bx (M\bx)_i \right| &= n^{-1} \sqrt{d} \left|\sum_{i=1}^d s_i^\bx (M\bx)_i \right| = O\left(n^{-1} \sqrt{d \log d} \right) = o_d(1).
    \end{align*}
\end{proof}

\medskip
We turn our attention to the left hand side of~\eqref{eq:new-goal}. 

\begin{lemma} \label{PxyBeta}
Let $n = \omega(d)$ be an even integer and let $\delta = \sqrt{d/n}$. Let $\bx, \by \in \supp \cS$ satisfying $- 1/2 \leq \varepsilon := \varepsilon(\bx, \by) \leq 1/2$ be given, and let 
\begin{align*}
\beta(\bx,\by) := &\exp \left(d\epsilon^2 + 2 \delta^2 | \epsilon \langle M\bx , M\by \rangle |+ 2 n^{-1} \sqrt{d} \left|\epsilon \la M \bx, s^{\by} \ra \right| + 2 n^{-1} \sqrt{d} \left|\epsilon \la M \by, s^{\bx} \ra \right| \right) \\
&\cdot \exp \left(\left(O \left(n^{-2} d^{3/2} \log d \right) + o \left(n^{-1} \sqrt{d} \right) \right) \left( |\la M \bx, \bold{1} \ra| + |\la M \by, \bold{1} \ra| \right) \right) .
\end{align*}
Then
$$
    P_{\bx,\by} \leq (1+o_d(1))\cdot P_{\bx}P_{\by} \cdot \beta(\bx,\by) \cdot \exp \left(- \delta_{\bx} - \delta_{\by} \right).
$$
\end{lemma}

\begin{proof}
Owing to our assumption that $|\varepsilon|\leq 1/2$, we may restrict our attention to pairs $(\bx,\by) \in (\supp \cS)^2$ such that $\bx \neq \by$. Given such a pair, let $k_i^\bx$ and $k_i^\by$ be integers satisfying $t_i^\bx = 2k_i^\bx$ and $t_i^\by = 2k_i^\by$. In a manner similar to that seen in the proof of Lemma~\ref{lemma:Px1}, it holds that
\begin{align}
P_{\bx,\by} & = \Pr_R\left[R\bx = (t_1^{\bx}, \ldots, t_d^{\bx}), R\by = (t_1^{\by}, \ldots, t_d^{\by}) \right] =  \prod_{i=1}^d \Pr_R\left[(R\bx)_i = t_i^\bx ,\; (R\by)_i = t_i^\by  \right] \nonumber \\
& = \prod_{i=1}^d \frac{1}{2^{n-1}} \binom{\alpha n}{\frac{\alpha n + k^\bx_i + k^\by_i}{2}} \binom{(1-\alpha)n}{\frac{(1-\alpha) n + k^\bx_i - k^\by_i}{2}}, \label{eq:Pxy-initial}
\end{align}
where the last equality holds by~\eqref{eq:joint-prob-1}. For
$$
L_i^{(1)} := \binom{\alpha n}{\frac{\alpha n + k^\bx_i + k^\by_i}{2}},
$$
we obtain
\begin{align*}
L_i^{(1)} &= \big(1+o_{n}(1)\big) 2^{\alpha n} \sqrt{\frac{2}{\pi \alpha n}} \exp\left(\Theta\left(\frac{(k_i^\bx + k_i^\by)^3}{(\alpha n)^2}\right) + o\left(\frac{k_i^\bx + k_i^{\by}}{\alpha n}\right) \right)  \exp \left(-  \frac{\left(k_i^\bx + k_i^\by \right)^2}{2 \alpha n} \right),
\end{align*}
where the equality holds by Proposition~\ref{prop:Spencer}. Similarly, for
$$
L_i^{(2)} :=  \binom{(1-\alpha)n}{\frac{(1-\alpha) n + k^\bx_i - k^\by_i}{2}},
$$
we obtain
$$
L_i^{(2)} = \big(1+o_{n}(1)\big) 2^{(1-\alpha) n} \sqrt{\frac{2}{\pi (1-\alpha) n}} \exp\left(\Theta\left(\frac{(k_i^\bx - k_i^\by)^3}{((1-\alpha) n)^2}\right) + o\left(\frac{k_i^\bx - k_i^{\by}}{(1 - \alpha) n}\right) \right) \exp \left(-  \frac{\left(k_i^\bx - k_i^\by \right)^2}{2 (1-\alpha) n} \right).
$$
It thus follows by~\eqref{eq:Pxy-initial} that
\begin{equation} \label{eq:Pxy-prior-Spencer-2}
P_{\bx,\by} = \prod_{i=1}^d \frac{1}{2^{n-1}} L_i^{(1)} L_i^{(2)}
    = (1+o_d(1))
    \prod_{i=1}^{d} \frac{4}{\pi n} \sqrt{\frac{1}{\alpha(1-\alpha)}}
    \cdot\exp\left(D_i+E_i+F_i \right),
\end{equation}
where 
\begin{align*}
D_i & := -\frac{(k_i^{\bx}+k_i^{\by})^2}{2 \alpha n} - \frac{(k_i^{\bx}-k_i^{\by})^2}{2 (1-\alpha) n}; \\ 
E_i &:= \Theta\left(\frac{(k_i^{\bx}+k_i^{\by})^3}{(\alpha n)^2} +\frac{(k_i^{\bx}-k_i^{\by})^3}{((1-\alpha) n)^2}\right); \\ 
F_i & := o\left(\frac{k_i^\bx + k_i^{\by}}{\alpha n} + \frac{k_i^\bx - k_i^{\by}}{(1 - \alpha) n}\right).
\end{align*}

In the sequel we prove that
\begin{align}
    \sum_{i=1}^{d}(D_i + E_i + F_i) \leq  &- \frac{1}{2n} \left( \sum_{i=1}^{d} (t_i^\bx)^2 + \sum_{i=1}^{d} (t_i^\by)^2 \right)
    + 2 \delta^2 | \epsilon \langle M\bx , M\by \rangle | \nonumber \\
    &+ 2 n^{-1} \sqrt{d} \left|\epsilon \la M \bx, s^{\by} \ra \right| + 2 n^{-1} \sqrt{d} \left|\epsilon \la M \by, s^{\bx} \ra \right| \nonumber  \\
    &+ O \left(n^{-2} d^{3/2} \log d \right) \left( |\la M \bx, \bold{1} \ra| +  |\la M \by, \bold{1} \ra| \right) \nonumber \\
    &+ o \left(n^{-1} \sqrt{d} \right) \left( |\la M \bx, \bold{1} \ra| + |\la M \by, \bold{1} \ra| \right) + o(1). \label{eq:DiEiFi}
\end{align}
Using~\eqref{eq:Pxy-prior-Spencer-2} and~\eqref{eq:DiEiFi} the proof concludes as follows. First, note that
\begin{align*}
    P_{\bx,\by} &= (1+o_d(1))
    \prod_{i=1}^{d} \frac{4}{\pi n} \sqrt{\frac{1}{\alpha(1-\alpha)}}
    \cdot\exp\left(D_i + E_i + F_i\right) \\
    &\leq (1+o_d(1))
     \left(\frac{8}{\pi n}\right)^d
     \exp\left(d\epsilon^2 + \sum_{i=1}^{d}(D_i + E_i + F_i) \right)\\
     &\leq (1+o_d(1))
     \left(\frac{8}{\pi n}\right)^d
     \exp\left( - \frac{1}{2n} \left( \sum_{i=1}^{d} (t_i^\bx)^2 + \sum_{i=1}^{d} (t_i^\by)^2 \right) \right)
    \cdot \beta(\bx,\by),
\end{align*}
where the first inequality holds by Claim~\ref{claim:abc}(b).

\medskip
Second, note that since $P_\bx =  (1+o_d(1))\left(\frac{8}{\pi n}\right)^{d/2} \cdot
    \exp\left( - \frac{1}{2n} \sum_{i=1}^{d} (t_i^\bx)^2 \right) \cdot \exp \left(\delta_{\bx} \right)$ holds by Lemma~\ref{lemma:Px1}, it follows that     
\begin{align*}
    P_{\bx,\by}
     &\leq (1+o_d(1)) P_{\bx} P_{\by} \cdot \beta(\bx, \by) \cdot \exp \left(- \delta_{\bx} - \delta_{\by} \right),
\end{align*}
as claimed. It remains to prove~\eqref{eq:DiEiFi}; to do so, we estimate each of the sums $\sum D_i$, $\sum E_i$, and $\sum F_i$ separately.

\bigskip
\noindent
{\bf Estimating $\boldsymbol{\sum D_i}$.} Start by writing 
\begin{align}
    D_i &= -\frac{(1-\alpha)(k_i^{\bx}+k_i^{\by})^2
    + \alpha(k_i^{\bx}-k_i^{\by})^2}{2 \alpha(1-\alpha) n} \nonumber
    \\
    &= -\frac{4 (1-\alpha)
    \left((k_i^{\bx})^2 + 2k_i^{\bx}k_i^{\by} + (k_i^{\by})^2 \right)
    + 4 \alpha
    \left((k_i^{\bx})^2 - 2k_i^{\bx}k_i^{\by} + (k_i^{\by})^2 \right)}
    {2 (1-\epsilon^2) n } \nonumber
    \\
    &= -\frac{4 \left((k_i^{\bx})^2 + (k_i^{\by})^2\right)}{2 (1-\epsilon^2) n}
    - \frac{4 (1-2\alpha) \cdot 2k_i^{\bx}k_i^{\by}}{2 (1-\epsilon^2 )n} \nonumber
    \\
    &= -\frac{(t_i^{\bx})^2 + (t_i^{\by})^2}{2n}
    - \frac{\epsilon^2\left((t_i^{\bx})^2 + (t_i^{\by})^2\right)}{2(1-\epsilon^2)n}
    + \frac{\epsilon t_i^{\bx}t_i^{\by}}{(1-\epsilon^2) n} \nonumber\\
    &\leq -\frac{(t_i^{\bx})^2 + (t_i^{\by})^2}{2n}
    + \frac{\epsilon t_i^{\bx}t_i^{\by}}{(1-\epsilon^2)n},\label{eq:Di2}
\end{align}
 where the second equality holds by Claim~\ref{claim:abc}(b); the last equality holds since $(1 - 2\alpha) = - \varepsilon$, $t_i^{\bx} = 2 k_i^{\bx}$ and $t_i^{\by} = 2 k_i^{\by}$, and by the following equality 
$$
\frac{A}{(1-z)B}=\frac{A}{B}+\frac{zA}{(1-z)B},
$$
holding for every $A,B$ and $z$. 
Finally, the inequality follows by discarding the middle term appearing on the preceding line; the latter is negative owing to $(1-\epsilon^2) \geq 1/2$, which holds since $|\epsilon| \leq 1/2$.

For the term $t_i^{\bx}t_i^{\by}$ appearing on the right hand side of~\eqref{eq:Di2}, we may write 
\begin{align*}
    t_i^{\bx}t_i^{\by} &= \left(-\sqrt{d}(M\bx)_i + s_i^\bx \right) \left(-\sqrt{d}(M\by)_i + s_i^\by \right)  \\
    &= d(M\bx)_i(M\by)_i - s_i^\by\sqrt{d}(M\bx)_i - s_i^\bx\sqrt{d}(M\by)_i + s_i^\bx s_i^\by;
\end{align*}
indeed, for every $i \in [d]$ and $\bz \in \{\bx, \by\}$, the equality $t_i^{\bz} = - \sqrt{d}(M\bz)_i + s_i^\bz$ holds, where $s_i^\bz \in [-4,4]$. Set 
$$
\mathcal{N} := \sum_{i=1}^{d}\frac{\sqrt{d} \epsilon s^{\by}_i (M\bx)_i}{(1-\epsilon^2)n}
    + \sum_{i=1}^{d}\frac{\sqrt{d} \epsilon s^{\bx}_i (M\by)_i}{(1-\epsilon^2)n} - \sum_{i=1}^{d}\frac{\epsilon s^{\bx}_i s^{\by}_i}{(1-\epsilon^2)n}
$$
and note that 
\begin{align*}
    \sum_{i=1}^{d} D_i &\leq
     - \sum_{i=1}^{d} \frac{(t_i^\bx)^2
    + (t_i^\by)^2}{2n}
    + \sum_{i=1}^{d} \frac{d \epsilon (M\bx)_i (M\by)_i}{(1-\epsilon^2)n}
    - \mathcal{N}
    \\
    &\leq
    - \frac{1}{2n} \left( \sum_{i=1}^{d} (t_i^\bx)^2 + \sum_{i=1}^{d} (t_i^\by)^2 \right)
    + 2 \delta^2 | \epsilon \langle M\bx , M\by \rangle | + |\mathcal{N}| 
\end{align*}
then holds, where in the last inequality we use the fact that $(1-\epsilon^2) \geq 1/2$. Additionally,
\begin{align*}
    |\mathcal{N}| &\leq \frac{\sqrt{d}}{(1 - \epsilon^2) n} \left|\sum_{i=1}^d \epsilon s^{\by}_i (M \bx)_i \right| + \frac{\sqrt{d}}{(1 - \epsilon^2) n} \left|\sum_{i=1}^d \epsilon s^{\bx}_i (M \by)_i \right| + \sum_{i=1}^{d}\frac{16 |\epsilon|}{(1-\epsilon^2)n} \\ 
    &\leq 2 n^{-1} \sqrt{d} \left|\epsilon \la M \bx, s^{\by} \ra \right| + 2 n^{-1} \sqrt{d} \left|\epsilon \la M \by, s^{\bx} \ra \right| + O \left(d n^{-1} \right) \\
    &\leq 2 n^{-1} \sqrt{d} \left|\epsilon \la M \bx, s^{\by} \ra \right| + 2 n^{-1} \sqrt{d} \left|\epsilon \la M \by, s^{\bx} \ra \right| + o_d(1),
\end{align*}
where the second inequality holds since $|\epsilon| \leq 1/2$ and thus $(1-\epsilon^2) \geq 1/2$; the last inequality holds since $n = \omega(d)$ by the premise of the lemma; the constant $16$ comes about from the fact that $s_i^\bz \in [-4,4]$ upheld by definition for every $i \in [d]$ and every $\bz \in \{\bx,\by\}$. 

\bigskip
We conclude that
\begin{align} \label{eq::sumDi}
    \sum_{i=1}^d D_i &\leq - \frac{1}{2n} \left( \sum_{i=1}^{d} (t_i^\bx)^2 + \sum_{i=1}^{d} (t_i^\by)^2 \right)+ 2 \delta^2 | \epsilon \langle M\bx , M\by \rangle | \nonumber\\ 
    &+ 2 n^{-1} \sqrt{d} \left|\epsilon \la M \bx, s^{\by} \ra \right| + 2 n^{-1} \sqrt{d} \left|\epsilon \la M \by, s^{\bx} \ra \right| + o_d(1).
\end{align}

\medskip\noindent
{\bf Estimating $\boldsymbol{\sum E_i}$.} Start by writing 
\begin{align}
        \sum_{i=1}^d \left(\frac{(k_i^{\bx} + k_i^{\by})^3}{(\alpha n)^2} + \frac{(k_i^{\bx} - k_i^{\by})^3}{((1-\alpha) n)^2} \right)
        &\leq \frac{c_1 \left| \sum_{i=1}^d (t_i^{\bx})^3 \right| + c_2 \left| \sum_{i=1}^d (t_i^{\by})^3 \right|}{n^2} \nonumber \\
        &+ \frac{c_3 \left|\sum_{i=1}^d (t_i^{\bx})^2 t_i^{\by} \right| + c_4 \left|\sum_{i=1}^d t_i^{\bx} (t_i^{\by})^2 \right|}{n^2},
        \label{eq:Ei1}
\end{align}
where $c_1, c_2, c_3,$ and $c_4$ are some constants, depending only on $\alpha$ (recall that $1/4 \leq \alpha \leq 3/4$ holds by Claim~\ref{claim:abc}(a) and since $|\varepsilon| \leq 1/2$ by the premise of the lemma). As in~\eqref{eq::sumtix3}, it follows by Lemma~\ref{lem:concentrationS}(ii) that 
$$
\left|\sum_{i=1}^d (t^{\bx}_i)^3 \right| \leq d^{3/2} \|M \bx\|_{\infty}^2 |\la M \bx, \bold{1} \ra| + o(n^2) = O \left(d^{3/2} \log d \cdot |\la M \bx, \bold{1} \ra| \right) + o(n^2)
$$
and, similarly,
$$
\left|\sum_{i=1}^d (t^{\by}_i)^3 \right| \leq d^{3/2} \|M \by\|_{\infty}^2 |\la M \by, \bold{1} \ra| + o(n^2) = O \left(d^{3/2} \log d \cdot |\la M \by, \bold{1} \ra| \right) + o(n^2).
$$
An analogous argument shows that
$$
\left|\sum_{i=1}^d (t^{\bx}_i)^2 t^{\by}_i \right| \leq d^{3/2} \|M \bx\|_{\infty}^2 |\la M \by, \bold{1} \ra| + o(n^2) = O \left(d^{3/2} \log d \cdot |\la M \by, \bold{1} \ra| \right) + o(n^2)
$$
and
$$
\left|\sum_{i=1}^d t^{\bx}_i (t^{\by}_i)^2 \right| \leq d^{3/2} \|M \by\|_{\infty}^2 |\la M \bx, \bold{1} \ra| + o(n^2) = O \left(d^{3/2} \log d \cdot |\la M \bx, \bold{1} \ra| \right) + o(n^2).
$$

It thus follows by~\eqref{eq:Ei1} that
\begin{align} \label{eq:t3l}
    \sum_{i=1}^{d} E_i &= \sum_{i=1}^d \Theta \left(\frac{(k_i^{\bx} + k_i^{\by})^3}{(\alpha n)^2} + \frac{(k_i^{\bx} - k_i^{\by})^3}{((1-\alpha) n)^2} \right) \nonumber \\
    &= O \left(n^{-2} d^{3/2} \log d \right) \left( |\la M \bx, \bold{1} \ra| +  |\la M \by, \bold{1} \ra| \right) + o(1).
\end{align}

\medskip\noindent
{\bf Estimating $\boldsymbol{\sum F_i}$.} It follows by~\eqref{eq::sumtix} that  
$$
\left| \sum_{i=1}^d t_i^\bx \right| \leq \sqrt{d} |\la M \bx, \bold{1} \ra| + o(n) \quad \text{and}\quad \left| \sum_{i=1}^d t_i^\by \right| \leq \sqrt{d} |\la M \by, \bold{1} \ra| + o(n).
$$
both hold. 
Hence,
\begin{align} \label{eq::Fi}
\sum_{i=1}^d F_i &= \sum_{i=1}^d o\left(\frac{k_i^\bx + k_i^{\by}}{\alpha n} + \frac{k_i^\bx - k_i^{\by}}{(1 - \alpha) n}\right) = o \left(\frac{\left| \sum_{i=1}^d t_i^\bx \right|}{n} \right) + o \left(\frac{\left| \sum_{i=1}^d t_i^\by \right|}{n} \right) \nonumber \\
&= o \left(n^{-1} \sqrt{d} \right) \left( |\la M \bx, \bold{1} \ra| + |\la M \by, \bold{1} \ra| \right) + o(1).
\end{align}

\bigskip
We conclude the proof by noticing that combining~\eqref{eq::sumDi}, ~\eqref{eq:t3l}, and~\eqref{eq::Fi} implies~\eqref{eq:DiEiFi}.
\end{proof}

\medskip

Following Lemma~\ref{PxyBeta}, we turn to estimate the terms $\exp \left(- \delta_{\bx} - \delta_{\by} \right)$ and $\beta(\bx, \by)$; in fact, estimations of these in expectation suffices. In the sequel, we prove the following two lemmas. 

\begin{lemma} \label{lem::expDelta}
If $n=\omega(d\log d)$, then
$$
\mathbb{E}_{\bx \sim \cS} \left[\exp \left(\delta_{\bx} \right) \right] = 1 + o_d(1).
$$
\end{lemma}

\begin{lemma} \label{lemma:ExpBeta}
If $n=\omega(d\log d)$, then
    $$
        \mathbb{E}_{\bx,\by \sim \cS}[\beta(\bx,\by)] = 1+o_d(1).
    $$
\end{lemma}

\noindent
Postponing the proofs of Lemmas~\ref{lem::expDelta} and~\ref{lemma:ExpBeta} until later, we first deduce~\eqref{eq:new-goal} from these and thus conclude the proof of Claim~\ref{clm:main}; this deduction is captured in the following lemma.   

\begin{lemma}
    If $n = \omega(d \log d)$, then
    $$
    \mathbb{E}_{\bx,\by\sim \cS}[P_{\bx,\by}] \leq (1+o_d(1))\mathbb{E}_{\bx,\by\sim \cS}[P_\bx P_\by].
    $$
\end{lemma}

\begin{proof}
    First, note that
    \begin{align*}
        \mathbb{E}_{\bx,\by\sim \cS}[P_{\bx}P_{\by}]
        &= \mathbb{E}_{\bx,\by\sim \cS} \left[(1 + o_d(1)) p^2 \exp \left(\delta_{\bx} + \delta_{\by}\right) \right]
        = (1 + o_d(1)) p^2 \cdot \mathbb{E}_{\bx,\by\sim \cS} \left[\exp \left(\delta_{\bx} \right) \cdot \exp \left(\delta_{\by}\right) \right] \\
        &= (1 + o_d(1)) p^2 \cdot \mathbb{E}_{\bx,\by\sim \cS} \left[\exp \left(\delta_{\bx} \right) \right] \cdot  \mathbb{E}_{\bx,\by\sim \cS} \left[\exp \left(\delta_{\by}\right) \right] \geq (1 - o_d(1)) p^2,
    \end{align*}
    where the first equality holds by Lemma~\ref{lemma:Pxeqp}, the last equality holds since $\bx$ and $\by$ are sampled independently and thus $\exp \left(\delta_{\bx} \right)$ and $\exp \left(\delta_{\by} \right)$ are independent, and the inequality holds by Lemma~\ref{lem::expDelta}.
   
    \bigskip
    Next, let $\mathcal{E}$ denote the event that $|\varepsilon| > 1/2$.
    Since $p$ is fixed, it follows that
    \begin{align} \label{eq:F2}
        \mathbb{E}_{\bx,\by\sim \cS}[P_{\bx,\by}] &\leq
        \mathbb{E}_{\bx,\by\sim \cS}[P_{\bx,\by} | \bar{\mathcal{E}}]
        + \mathbb{E}_{\bx,\by\sim \cS}[P_{\bx,\by} | \mathcal{E}] \cdot \Pr_{\bx,\by\sim \cS}[\mathcal{E}] \nonumber \\
        &\leq \mathbb{E}_{\bx,\by\sim \cS}[P_{\bx,\by} | \bar{\mathcal{E}}]
        + \mathbb{E}_{\bx,\by\sim \cS}[(1+o_d(1)) p \cdot \exp(\delta_{\bx})|\ \mathcal{E}] \cdot \Pr_{\bx,\by\sim \cS}[\mathcal{E}]
        \nonumber \\
        &= \mathbb{E}_{\bx,\by\sim \cS}[P_{\bx,\by} | \bar{\mathcal{E}}]
        + (1+o_d(1)) p^2 \cdot p^{-1} \cdot \mathbb{E}_{\bx \sim \cS}[\exp(\delta_{\bx})] \cdot \Pr_{\bx,\by\sim \cS}[\mathcal{E}] \nonumber\\
        &\leq \mathbb{E}_{\bx,\by\sim \cS}[P_{\bx,\by} | \bar{\mathcal{E}}]
        + (1+o_d(1)) p^2 \cdot p^{-1} \cdot \Pr_{\bx,\by\sim \cS}[\mathcal{E}] 
    \end{align}
    where the second inequality holds since $P_{\bx,\by} \leq \min \{P_\bx, P_\by\} \leq (1+o_d(1)) p \cdot \exp \left(\delta_{\bx} \right)$ by Lemma~\ref{lemma:Pxeqp}, the equality holds since, once $\by$ becomes irrelevant, so does the event $\mathcal{E}$, and the last inequality holds by Lemma~\ref{lem::expDelta}.
   
    \medskip
    Given any vector $\by \in \{-1, 1\}^n$, note that $n^{-1/2} \by \in \mathbb{S}^{n-1}$. It thus follows by Lemma~\ref{lem:concentrationS}(iii) that
    \begin{align*}
        \Pr_{\bx\sim \cS} \left[|n^{-1} \langle \bx, \by \rangle| > 1/2 \right]
        &= \Pr_{\bx\sim \cS} \left[\left|\left\langle \bx, n^{-1/2}\by \right\rangle \right| > \sqrt{n}/2 \right] \\
        &\leq d^{C_{\ref{lem:concentrationS}}} \cdot \exp\left(- n/36 \right) \leq \exp \left(- n/40 \right),
    \end{align*}
    holds for every $\by \in \{-1,1\}^n$. It then follows by the law of total probability that
    \begin{equation} \label{eq::Eexp-n}
         \Pr_{\bx,\by\sim \cS}[\mathcal{E}] = \Pr_{\bx,\by \sim \cS}[|\varepsilon|>1/2] \leq \exp(- \Theta(n)).
    \end{equation}
   
    On the other hand, since $p = \left(\frac{8}{\pi n}\right)^{d/2}\cdot\exp\left(- r^2 \delta^2/2 \right)$ and $n = \omega \left(d \log d \right)$,
    it follows that
    \begin{align} \label{eq::p1}
        p^{-1} = \left(\frac{\pi n}{8}\right)^{d/2}\cdot \exp(r^2 \delta^2/2)
        \leq \exp\left( O\left(d \log n\right) + O\left(d^2/n\right) \right)
        = \exp\left( o(n) \right).
    \end{align}
    It then follows by \eqref{eq:F2}, \eqref{eq::Eexp-n} and~\eqref{eq::p1} that
    \begin{align} \label{eq::PxyNotE}
        \mathbb{E}_{\bx,\by\sim \cS}[P_{\bx,\by}] &\leq
        \mathbb{E}_{\bx,\by\sim \cS}[P_{\bx,\by} | \bar{\mathcal{E}}]
        + p^2 \cdot o_n(1).
    \end{align}

    \bigskip
    Conditioning on $\bar{\mathcal{E}}$ and recalling that $n = \omega \left(d \log d \right)$ holds by the premise of the lemma, we obtain
    \begin{equation} \label{eq::conditionNotE}
        P_{\bx,\by} \leq (1+o_d(1)) P_\bx P_\by \cdot \beta(\bx,\by) \cdot \exp \left(- \delta_{\bx} - \delta_{\by} \right)
        = (1+o_d(1)) \cdot p^2 \cdot \beta(\bx, \by),
    \end{equation}
    where the inequality holds by Lemma~\ref{PxyBeta} and the equality holds by Lemma~\ref{lemma:Pxeqp}.
   
    Combining~\eqref{eq::PxyNotE} and~\eqref{eq::conditionNotE} we conclude that
        \begin{align*}
        \mathbb{E}_{\bx,\by\sim \cS}[P_{\bx,\by}] &\leq
        \mathbb{E}_{\bx,\by\sim \cS}[(1 + o_d(1)) \cdot p^2 \cdot \beta(\bx,\by) | \bar{\mathcal{E}}]
        + p^2 \cdot o_n(1) \\
        &\leq
        (1 + o_d(1)) \cdot p^2 \cdot \mathbb{E}_{\bx,\by\sim \cS}[\beta(\bx,\by) | \bar{\mathcal{E}}]
        + p^2 \cdot o_n(1) \\
        &\leq
        (1 + o_d(1)) \cdot p^2 \cdot \mathbb{E}_{\bx,\by\sim \cS}[\beta(\bx,\by)]
        + p^2 \cdot o_n(1) \\
        &= (1 + o_d(1)) \cdot p^2 + p^2 \cdot o_n(1)
        \label{eq:Fin1} \\
        &= (1+o_d(1))\cdot p^2,
    \end{align*}
    where the last inequality holds since $\mathbb{P}[\bar{\mathcal{E}}] = 1 - o(1)$ and $\beta$ is non-negative, and the first equality holds by Lemma~\ref{lemma:ExpBeta}.
\end{proof}

It remains to prove Lemmas~\ref{lem::expDelta} and~\ref{lemma:ExpBeta}. The following result facilitates our proofs of these lemmas; proof of the latter is essentially that seen for~\cite[Lemma~2.4]{BJMSS} and is thus omitted. 

\begin{lemma} \label{lemma:ExpX} {\em~\cite[Lemma~2.4]{BJMSS}}
    Let $X$ be a non-negative random variable which satisfies
    $$
        \Pr[X>t] \leq d^{\xi_1} \cdot e^{- t^2/\xi_2} \textrm{ for any } t > 0,
    $$
    for some positive constants $\xi_1$ and $\xi_2$. Then, for any $\lambda = \kappa \sqrt{\log d}$, where $\kappa \geq 2 \sqrt{\xi_1 \xi_2}$, we have
    $$
       \mathbb{E} \left[\exp \left(X^2/\lambda^2 \right) \right] \leq 1 + 4 \xi_1 \xi_2/\kappa^2 + o_d(1).
    $$
\end{lemma}
\begin{proofof}{Lemma~\ref{lem::expDelta}}
Starting with the upper bound, given any $\bx \in \cS$, let
$$
Z := Z(\bx) = \left(c n^{-2} d^{3/2} \log d+ \zeta n^{-1} \sqrt{d} \right) |\la M \bx, \bold{1} \ra|,
$$
where $c$ is a constant and $\zeta = o(1)$ are chosen so as to ensure that $\left|\delta_{\bx} \right| \leq Z$ holds. If $\left| \left\la M\bx, d^{-1/2} \bold{1} \right\ra \right| < 1$, then the assumption that $n = \omega(d \log d)$ implies that $Z = o_d(1)$ and thus
$$
\mathbb{E}_{\bx \sim \cS} \left[\exp \left(\delta_{\bx} \right) \right] \leq \mathbb{E}_{\bx \sim \cS} \left[\exp \left(Z \right) \right]  \leq 1 + o_d(1).
$$
Assume then that $\left| \left\la M\bx, d^{-1/2} \bold{1} \right\ra \right| \geq 1$. In this case
\begin{align*} 
Z = \left(c n^{-2} d^2 \log d+ \zeta n^{-1} \sqrt{d} \right) \left| \left\la M \bx, d^{- 1/2} \bold{1} \right\ra \right| \leq \left(c n^{-2} d^2 \log d+ \zeta n^{-1} \sqrt{d} \right) \left| \left\la M \bx, d^{- 1/2} \bold{1} \right\ra \right|^2.
\end{align*}
Writing $\gamma := \left\la M \bx, d^{-1/2} \bold{1} \right\ra$ and $\lambda := \sqrt{n/d}$, it then follows that
$$
Z \leq \left(c n^{-2} d^2 \log d + \zeta n^{-1} d \right) |\gamma|^2 \leq n^{-1} d |\gamma|^2 = |\gamma|^2/\lambda^2,
$$
where the second inequality holds since $\zeta = o(1)$ and since $n = \omega(d \log d)$ implies that $n^{-2} d^2 \log d = o \left(n^{-1} d \right)$. 

Then,
\begin{align*}
\Pr_{\bx\sim \cS} [|\gamma| > t] &= \Pr_{\bx\sim \cS} \left[|\langle M \bx, d^{-1/2} \bold{1} \rangle| > t \right] \leq d^{C_{\ref{lem:concentrationS}}} \exp \left(- t^2/8 \right)
\end{align*}
holds for any $t > 0$, by Lemma~\ref{lem:concentrationS}(iii). We conclude that the non-negative random variable $X := |\gamma|$ satisfies the conditions of Lemma~\ref{lemma:ExpX}, implying that
\begin{equation}\label{eq::EdeltaXUB_1} 
\mathbb{E}_{\bx \sim \cS} \left[\exp \left(\delta_{\bx} \right) \right] \leq  \mathbb{E}_{\bx \sim \cS} \left[\exp \left(\left|\delta_{\bx} \right| \right) \right] \leq \mathbb{E}_{\bx,\by \sim \cS}[\exp(Z)] \leq \mathbb{E}_{\bx,\by \sim \cS} \left[\exp \left(X^2/\lambda^2 \right) \right] = 1+o_d(1), 
\end{equation}
where the last equality holds since $n = \omega(d \log d)$ by the premise of the lemma and thus $\lambda = \omega(\sqrt{\log d})$.

\medskip
Next, we prove that $\mathbb{E}_{\bx, \by \sim \cS} \left[\exp \left(\delta_{\bx} \right) \right] \geq 1 - o_d(1)$. 
Let $Y = \exp \left(\left|\delta_{\bx} \right| \right)$; note that $Y$ is positive. Let $g : \mathbb{R}^+ \to \mathbb{R}^+$ be defined by $g(x) = 1/x$; note that $g$ is convex. It thus follows by Jensen's inequality\footnote{Jensen's inequality, see, e.g.~\cite{Durrett}, asserts that $g(\mathbb{E}(X)) \leq \mathbb{E}(g(X))$ holds whenever $X$ is a random variable and $g$ is a convex function.} that
\begin{align*}
\mathbb{E}_{\bx \sim \cS} \left[\exp \left(\delta_{\bx} \right) \right] &\geq \mathbb{E}_{\bx \sim \cS} \left[\exp \left(- \left|\delta_{\bx} \right| \right) \right] = \mathbb{E}_{\bx \sim \cS} [g(Y)] \geq g \left(\mathbb{E}_{\bx \sim \cS} [Y] \right) \\
&= \frac{1}{\mathbb{E}_{\bx \sim \cS} \left[\exp \left(\left|\delta_{\bx} \right| \right) \right]} \geq \frac{1}{1 + o_d(1)} = 1 - o_d(1),
\end{align*}
where the last inequality holds by~\eqref{eq::EdeltaXUB_1}.
\end{proofof}

\begin{proofof}{Lemma~\ref{lemma:ExpBeta}}
Let
$$
Z_1 := Z_1(\bx, \by) = d\epsilon^2 + 2 \delta^2 |\epsilon \langle M\bx, M\by \rangle| + 2 n^{-1} \sqrt{d} \left|\epsilon \la M \bx, s^{\by} \ra \right| + 2 n^{-1} \sqrt{d} \left|\epsilon \la M \by, s^{\bx} \ra \right|
$$
and let
$$
Z_2 := Z_2(\bx, \by) = \left(c n^{-2} d^{3/2} \log d + \zeta n^{-1} \sqrt{d} \right) \left( |\la M \bx, \bold{1} \ra| + |\la M \by, \bold{1} \ra| \right),
$$
where $c$ is a constant and $\zeta = o(1)$.

Set 
\begin{align*}
\bar\varepsilon &:= \bar\varepsilon(\bx, \by) = \left\langle \bx, n^{-1/2}\by \right\rangle, \\
\mu & := \mu(\bx, \by) = \left\langle M\bx, \|M \by\|_2^{-1} M\by \right\rangle,\\ 
\eta_1 &:= \eta_1(\bx, \by) = \left\la M \bx, \|s^{\by}\|_2^{-1} s^{\by} \right\ra,\\ 
\eta_2 &:= \eta_2(\bx, \by) = \left\la M \by, \|s^{\bx}\|_2^{-1} s^{\bx} \right\ra.
\end{align*}
Recalling that $\delta = \sqrt{d/n}$, 
we obtain
\begin{align} \label{eq::Z1}
    Z_1 &= \frac{d}{n} \cdot |\bar\varepsilon|^2 + 2 \frac{d \|M \by\|_2}{n^{3/2}} |\bar{\epsilon}| |\mu| + 2 \frac{\sqrt{d} \|s^{\by} \|_2}{n^{3/2}} |\bar{\epsilon}| |\eta_1| + 2 \frac{\sqrt{d} \|s^{\bx} \|_2}{n^{3/2}} |\bar{\epsilon}| |\eta_2| \nonumber \\ 
    &\leq \frac{d}{n} \cdot |\bar\varepsilon|^2 + O\left( \frac{d^{3/2}}{n^{3/2}} \right) |\bar{\epsilon}| |\mu| + 8 \frac{d}{n^{3/2}} |\bar{\epsilon}| |\eta_1| + 8 \frac{d}{n^{3/2}} |\bar{\epsilon}| |\eta_2| \nonumber \\
    &\leq \left( \frac{d}{n} + O \left(\frac{ d^{3/2}}{n^{3/2}} \right) + \frac{16 d}{n^{3/2}} \right) (|\bar\varepsilon| + |\mu| + |\eta_1| + |\eta_2|)^2 \nonumber \\ 
    &\leq 2 n^{-1} d (|\bar\varepsilon| + |\mu| + |\eta_1| +  |\eta_2|)^2, 
\end{align}
where the first inequality holds by Lemma~\ref{lem:concentrationS}(i) and by Lemma~\ref{lem::choosesX}(1), and the last inequality holds since $n = \omega(d)$.


Let $\gamma = \left|\left\la M \bx, d^{- 1/2} \bold{1} \right\ra \right| + \left|\left\la M \by, d^{- 1/2} \bold{1} \right\ra \right|$. If $\gamma < 1$, then $Z_2 = o_d(1)$; otherwise
\begin{align} \label{eq::Z2}
Z_2 &= \left(c n^{-2} d^2 \log d + \zeta n^{-1} d \right) \left( \left| \left\la M \bx, d^{- 1/2} \bold{1} \right\ra \right| + \left| \left\la M \by, d^{- 1/2} \bold{1} \right\ra \right| \right) \nonumber \\ 
&= O \left(n^{-2} d^2 \log d + \zeta n^{-1} d \right) \gamma^2 \leq n^{-1} d \gamma^2,
\end{align}
where the last inequality holds since $\zeta = o(1)$ and since $n = \omega(d \log d)$ implies that $n^{-2} d^2 \log d = o \left(n^{-1} d \right)$. 

Let $\lambda := \sqrt{n/(2d)}$. Combining~\eqref{eq::Z1} and~\eqref{eq::Z2} we obtain
\begin{align} \label{eq::Z1plusZ2}
Z_1 + Z_2 &\leq 2 n^{-1} d (|\bar\varepsilon| + |\mu| + |\eta_1| + |\eta_2|)^2 + n^{-1} d \gamma^2 \leq 2 n^{-1} d (|\bar\varepsilon| + |\mu| + |\eta_1| + |\eta_2| + \gamma)^2 \nonumber \\
&\leq (|\bar\varepsilon| + |\mu| + |\eta_1| + |\eta_2| + \gamma)^2/\lambda^2.
\end{align}

Given any vector $\by \in \{-1, 1\}^n$, note that $n^{-1/2} \by \in \mathbb{S}^{n-1}$. It thus follows by Lemma~\ref{lem:concentrationS}(iii) that
\begin{align*}
\Pr_{\bx\sim \cS} \left[\left| \left\langle \bx, n^{-1/2}\by \right\rangle \right| > t \right] \leq  d^{C_{\ref{lem:concentrationS}}} \exp \left(- t^2/9 \right)
\end{align*}
holds for any $t > 0$. It then follows by the law of total probability that
\begin{equation} \label{eq::barEpsilon}
    \Pr[|\bar\varepsilon| > t] = \Pr_{\bx, \by \sim \cS} \left[|\langle \bx, n^{-1/2}\by \rangle| > t \right] \leq d^{C_{\ref{lem:concentrationS}}} \exp \left(-t^2/9 \right) \textrm{ for any } t > 0.
\end{equation}

Similarly, since $\|M \by\|_2^{-1} M\by \in \mathbb{S}^{d-1}$ and $\|s^{\by}\|_2^{-1} s^{\by} \in \mathbb{S}^{d-1}$ hold for every $\by \in \supp \cS$, $\|s^{\bx}\|_2^{-1} s^{\bx} \in \mathbb{S}^{d-1}$ holds for every $\bx \in \supp \cS$, and $d^{- 1/2} \bold{1} \in \mathbb{S}^{d-1}$, it follows that
\begin{align*}
    \Pr[|\mu| > t] &\leq d^{C_{\ref{lem:concentrationS}}} \exp \left(-t^2/9 \right) \textrm{ for any } t > 0 \\
    \Pr[|\eta_1| > t] &\leq d^{C_{\ref{lem:concentrationS}}} \exp \left(-t^2/9 \right) \textrm{ for any } t > 0 \\
    \Pr[|\eta_2| > t] &\leq d^{C_{\ref{lem:concentrationS}}} \exp \left(-t^2/9 \right) \textrm{ for any } t > 0 \\
    \Pr[\gamma > t] &\leq d^{C_{\ref{lem:concentrationS}}} \exp \left(-t^2/9 \right) \textrm{ for any } t > 0.
\end{align*}
We conclude that the non-negative random variable $X := |\bar\epsilon| + |\mu| + |\eta_1| + |\eta_2| + \gamma$ satisfies the conditions of Lemma~\ref{lemma:ExpX}, implying that
\begin{align*}
    \mathbb{E}_{\bx,\by \sim \cS}[\beta(\bx,\by)] = \mathbb{E}_{\bx,\by \sim \cS}[\exp(Z_1 + Z_2)] \leq \mathbb{E}_{\bx,\by \sim \cS} \left[\exp \left(X^2/\lambda^2 \right) \right] = 1+o_d(1),
\end{align*}
where the last equality holds since $n = \omega(d \log d)$ by the premise of the lemma and thus $\lambda = \omega(\sqrt{\log d})$.
\end{proofof}

\section{Concluding remarks} \label{sec::concluding}

We have proved that $\DISC(M + R/\sqrt{d}) = O(d^{-1/2})$ holds asymptotically almost surely, whenever $M \in \mathbb{R}^{d \times n}$ is Koml\'os, $R \in \mathbb{R}^{d \times n}$ is Rademacher, $d = \omega(1)$, and $n = \omega(d \log d)$. 

As stated by Bansal, Jiang, Meka, Singla, and Sinha~\cite[Section~3]{BJMSS}, considering other distributions for the entries of the random perturbation and specifically discrete ones is of high interest as well. In view of the aforementioned result in~\cite{ANW22} pertaining to the discrepancy of Bernoulli matrices, as well as the proclaimed $n = \omega(d^2)$ bound attained in~\cite{BJMSS} in the smoothed setting with Bernoulli noise, the following question seems to be a natural next step. 

\begin{question}\label{que:Ber}
Let $d =\omega(1)$ and $n = \omega(d \log d)$ be integers, and set $p:= p(n,d) >0$. Is it true that $\DISC(M +R) = O(1)$ holds a.a.s. whenever $M \in \mathbb{R}^{d \times n}$ is a Koml\'os matrix and $R \in \mathbb{R}^{d \times n}$ is a random matrix with each of its entries being an independent copy of $\Theta \left((pd)^{-1/2} \right)\mathrm{Ber}(p)$? 
\end{question}

\noindent
It is conceivable that for certain ranges of $p$, the $O(1)$ bound on the discrepancy, appearing in Question~\ref{que:Ber}, can be replaced with $1/\mathrm{poly}(d)$.

\bibliographystyle{amsplain}
\bibliography{Komlos-lit}

\appendix

\section{Proofs of Lemmas~\ref{lem:single-even} and~\ref{lem:joint-even}}\label{app:rud}

Prior to proving Lemmas~\ref{lem:single-even} and~\ref{lem:joint-even}, we collect several auxiliary results. 

\begin{observation}\label{obs:parity->diff-even}
Let $n \in \mathbb{N}$ be even and let $\bx,\by \in \{-1,1\}^n$ satisfying $\#_1(\bx) \equiv \#_1(\by) \pmod{2}$ be given. Then, $|\Diff(\bx,\by)|$ is even. 
\end{observation}

\begin{proof}
Let $A := A(\bx,\by) = |\{i \in [n] : \bx_i = \by_i = 1\}|$, let $B := B(\bx,\by) = |\{i \in [n] : \bx_i = \by_i = -1\}|$, let $C := C(\bx,\by) = |\{i \in [n] : \bx_i = 1, \by_i = -1\}|$, and let $D := D(\bx,\by) = |\{i \in [n] : \bx_i = -1, \by_i = 1\}|$. Suppose for a contradiction that $|\Diff(\bx,\by)|$ is odd. Since $|\Diff(\bx,\by)| = C + D$, we may assume without loss of generality that $C$ is even and $D$ is odd. Since, moreover, $n = A + B + C + D$ is even, we may further assume without loss of generality that $A$ is even and $B$ is odd. It then follows that $\#_1(\bx) = A + C$ is even, whereas $\#_1(\by) = A + D$ is odd; this contradicts the premise of the observation and concludes its proof.
\end{proof}

\begin{lemma}\label{lem:even-diff-same-sum}
Let $n \in \mathbb{N}$, let $k \in \mathbb{Z}$, and let $\bu,\bv \in \{-1,1\}^n$ be vectors satisfying $\sum_{i=1}^n \bu_i = 2k = \sum_{i=1}^n \bv_i$. Then, $|\Diff(\bv,\bu)|$ is even. 
\end{lemma}

\begin{proof}
Set
$$
O = \left\{ i\in \Diff(\bv,\bu): \bu_i =1\right\}\; \text{and}\; M = \left\{i \in \Diff(\bv,\bu): \bu_i = -1\right\}.
$$  
Then
\begin{align*}
2k &= \sum_{i=1}^n \bv_i  = \sum_{i \notin \Diff(\bv,\bu)} \bu_i + \sum_{i \in O} \left( \bu_i-2 \right) + \sum_{i \in M} \left( \bu_i+2 \right) \\
& = \sum_{i=1}^n \bu_i - 2|O| +2|M| = 2k - 2|O| + 2|M|.
\end{align*}
It follows that $|O| = |M|$, and thus $|\Diff(\bv,\bu)| = |O| + |M|$ is even.
\end{proof}

\begin{lemma}\label{lem:diff-even}
Let $\bu \in \{-1,1\}^n$ and let $\bv \in \mathcal{E}_n$. If $|\Diff(\bv,\bu)|$ is even, then $\bu \in \mathcal{E}_n$. 
\end{lemma}

\begin{proof}
The proof is via induction on $|\Diff(\bv,\bu)|$. If $|\Diff(\bv,\bu)| = 0$, then $\bu = \bv \in \mathcal{E}_n$. Suppose then that $|\Diff(\bv,\bu)| = 2$ and let $i,j \in [n]$ be the (sole) two distinct indices over which $\bu$ and $\bv$ differ. The equality $\#_1(\bu) = \#_1(\bv) - (\bv_i + \bv_j)$ coupled with the assumption that $\#_1(\bv)$ is even as well as the fact that $\bv_i + \bv_j \in \{-2,0,2\}$, imply that $\#_1(\bu)$ is even as well and thus $\bu \in \mathcal{E}_n$ as required.  

For the induction step, consider $\bv \in \mathcal{E}_n$ and $\bu \in \{-1,1\}^n$ satisfying $|\Diff(\bv,\bu)| = 2m +2$ for some positive integer $m$ and assume that the claim holds true for any pair of vectors $\bx \in \mathcal{E}_n$ and $\by \in \{-1,1\}^n$ satisfying $|\Diff(\bx,\by)| = 2k$ for some positive integer $k \leq m$. Let $1 \leq i < j \leq n$ be any two distinct indices for which $\bv_i \neq \bu_i$ and $\bv_j \neq \bu_j$ both hold. The vector
$$
\bv' := (\bv_1,\ldots,\bv_{i-1},-\bv_i,\bv_{i+1},\ldots,\bv_{j-1},-\bv_j,\bv_{j+1},\ldots,\bv_n)
$$
satisfies $|\Diff(\bv,\bv')| = 2$; hence, $\bv' \in \mathcal{E}_n$ holds by the induction hypothesis. Since, moreover, $|\Diff(\bu,\bv')| = 2m$, it follows by the induction hypothesis that $\bu \in \mathcal{E}_n$. This concludes the proof of the lemma. 
\end{proof}

We are now in position to prove the first main result of this section, namely Lemma~\ref{lem:single-even}. 

\begin{proofof}{Lemma~\ref{lem:single-even}}
A vector $\br \in \{-1, 1\}^n$ is said to be {\sl valid} if $\br \in \mathcal{E}_n$ and $\la \br,\bx \ra = 2k$. Since $|\mathcal{E}_n| = 2^{n-1}$, it suffices to prove that there are $\binom{n}{n/2+k}$ valid vectors. In light of~\eqref{eq:sum-2t}, it remains to prove that there is a bijection from the set of valid vectors to the set $S_k$. 

Given a valid vector $\br$ (such a vector exists by the premise of the lemma), define $\phi(\br) := (\br_1\bx_1,\ldots,\br_n\bx_n) \in \{-1,1\}^n$. The validity of $\br$ implies that $\sum_{i=1}^n \phi(\br)_i = 2k$ and thus $\phi(\br) \in S_k$. To see that $\phi(\cdot)$ is injective, note that given two different valid vectors $\br$ and $\br'$, there exists an index $i \in [n]$ such that $\br_i \neq \br'_i$. As $\bx$ is fixed, this compels that $\phi(\br)_i = \br_i \bx_i \neq \br'_i \bx_i = \phi(\br')_i$ so that $\phi(\br)\neq \phi(\br')$. 

To prove that $\phi(\cdot)$ is surjective, fix $\bv \in S_k$ and define the vector $\by \in \{-1,1\}^n$ whose entries are uniquely determined by the equalities $\bv_i = \by_i \bx_i$, that is, for every $i \in [n]$, if $\bv_i = \bx_i$, then $\by_i = 1$, and otherwise $\by_i = -1$. It is evident that, if $\by$ is valid, then $\bv = \phi(\by)$. Since, moreover, $\bv \in S_k$, it suffices to prove that $\by \in \mathcal{E}_n$. To that end, let $\br$ be an arbitrary valid vector. Since $\sum_{i=1}^n \phi(\br)_i = 2k = \sum_{i=1}^n \bv_i$, it follows by Lemma~\ref{lem:even-diff-same-sum} that $|\Diff(\bv,\phi(\br))|$ is even. Note that $\by_i = \br_i$ whenever $i \notin \Diff(\bv,\phi(\br))$, and $\by_i = -\br_i$ whenever $i \in \Diff(\bv,\phi(\br))$. Consequently, $|\Diff(\by,\br)|$ is even and thus $\by$ is even by Lemma~\ref{lem:diff-even}.
\end{proofof}

We conclude this section with a proof of Lemma~\ref{lem:joint-even}. 

\begin{proofof}{Lemma~\ref{lem:joint-even}}
Since $\#_1(\bx) \equiv \#_1(\by) \pmod{2}$ holds by assumption, it follows by Observation~\ref{obs:parity->diff-even} that $|\Diff(\bx,\by)| = 2m$ for some non-negative integer $m$. The set $\Diff(\bx,\by)$ having even cardinality has two useful implications. The first is that $n - |\Diff(\bx,\by)|$ is an even integer; this on account of $n$ being even by assumption. Using the previously introduced notation $\alpha n := \alpha(\bx, \by) n := n - |\Diff(\bx,\by)|$, we infer that $\alpha n$ and $(1-\alpha)n$ are both even integers.

The second implication is that $\la \bv, \bx \ra = \la \bv, \by \ra + \ell$, for some $\ell \in \{4k : k \in \mathbb{Z}, -m \leq k \leq m\}$, holds for every $\bv \in \{-1, 1\}^n$. Indeed, reaching $\la \bv, \bx \ra$ starting from $\la \bv, \by \ra$ entails iterating over each member of the even-sized set $\Diff(\bx,\by)$ and adding or subtracting two from the current value accumulated thus far. 


If, additionally, $\la \bv, \bx \ra = 2 k_{\bx}$ and $\la \bv, \by \ra = 2 k_{\by}$, where $k_{\bx}$ and $k_{\by}$ are integers, then $k_{\bx} \equiv k_{\by} \pmod{2}$, for indeed 
$$
k_{\bx} - k_{\by} = \frac{\la \bv, \bx \ra - \la \bv, \by \ra}{2} = \frac{\ell}{2} \in 2\mathbb{Z}.
$$

\medskip

Given $\bv \in \{-1,1\}^n$, set
$$
S_1(\bv) := \{i \in [n] \sm \Diff (\bx,\by) : \bv_i\bx_i = 1\}\; \text{and} \; S_2(\bv) := \{i \in \Diff(\bx,\by): \bv_i \bx_i =1\}.
$$
Additionally, set
$$
\bar S_1(\bv) := \Big([n]\sm \Diff(\bx,\by)\Big)\sm S_1(\bv)\; \text{and}\; \bar S_2(\bv) := \Diff(\bx,\by) \sm S_2(\bv).
$$ 
There exist integers $m_1 := m_1(\bv)$ and $m_2 := m_2(\bv)$ such that $|S_1(\bv)| = \frac{\alpha n}{2} + m_1$ and $|S_2(\bv)| = \frac{(1-\alpha)n}{2} + m_2$. If $\la \bv, \bx \ra = 2 k_{\bx}$ for some integer $k_{\bx}$, then 
$$
2 k_{\bx} = \sum_{i \in S_1(\bv)} 1 + \sum_{i \in \bar S_1(\bv)}(-1) + \sum_{i \in S_2(\bv)} 1 + \sum_{i \in \bar S_2(\bv)}(-1) = 2 m_1 + 2 m_2.
$$
Using the definition of $\Diff (\bx,\by)$, an analogous argument shows that if $\la \bv, \by \ra = 2 k_{\by}$ for some integer $k_{\by}$, then $2 k_{\by} = 2 m_1 - 2 m_2$.

Therefore\footnote{Recall that $k_{\bx} \equiv k_{\by} \pmod{2}$ so that $k_{\bx} \pm k_{\by}$ is even.} 
$$
m_1 = \frac{k_{\bx} + k_{\by}}{2} \; \text{ and } \; m_2 = \frac{k_{\bx} - k_{\by}}{2};
$$
in particular, $m_1$ and $m_2$ are independent of $\bv$. We conclude that the number of vectors $\br \in \mathcal{E}_n$ for which $\la \br,\bx \ra = 2 k_{\bx}$ and $\la \br,\by \ra = 2 k_{\by}$ both hold is
$$
\binom{\alpha n}{\frac{\alpha n}{2} + m_1}\binom{(1-\alpha)n}{\frac{(1-\alpha) n}{2} + m_2}.
$$
Since, moreover, $|\mathcal{E}_n| = 2^{n-1}$, it follows that
$$
\Pr \Big[\la \br,\bx \ra = 2 k_{\bx}, \la \br,\by \ra = 2 k_{\by} \Big] = \frac{1}{2^{n-1}}\binom{\alpha n}{\frac{\alpha n + k_{\bx} + k_{\by} }{2}}\binom{(1-\alpha)n}{\frac{(1-\alpha) n + k_{\bx} - k_{\by}}{2}}
$$
as claimed. 
\end{proofof}

\end{document}